\documentclass[article,11pt,numbers]{elsarticle}

\usepackage{graphicx} 
\usepackage{amsthm}
\usepackage{color}
\usepackage{tikz}
\usetikzlibrary {arrows.meta} 
\usepackage{changes}
\usepackage{comment}
\usepackage{enumitem}
\usepackage{amsmath,amssymb,url}

\usepackage{changes}
\usepackage{xcolor}

\numberwithin{equation}{section}

\theoremstyle{thmstyleone}
\newtheorem{theorem}{Theorem}[section]
\newtheorem{lemma}[theorem]{Lemma}
\newtheorem{proposition}[theorem]{Proposition}
\newtheorem{corollary}[theorem]{Corollary}
\newtheorem{claim}[theorem]{Claim}

\theoremstyle{definition}
\newtheorem{definition}[theorem]{Definition}

\newtheorem{example}[theorem]{Example}
\newtheorem{remark}[theorem]{Remark}
\newtheorem*{note}{Note} 
\newtheorem{problem}{Problem}

\newcommand{\map}{\longrightarrow}

\newcommand{\VV}{\mathsf{V}}
\newcommand{\WW}{\mathsf{W}}

\newcommand{\la}{\mathfrak{a}}
\newcommand{\lb}{\mathfrak{b}}
\newcommand{\lA}{\mathfrak{A}}

\newcommand{\alg}[1]{{\textbf{\upshape #1}}}  %

\newcommand{\HH}{\mathbf{H}}
\newcommand{\CSub}{\mathbf{S}}
\newcommand{\PP}{\mathbf{P}}
\newcommand{\II}{\mathbf{I}}

\newcommand{\ZZ}{\alg Z}

\newcommand{\KK}{\mathsf{K}}

\newcommand{\zero}{\mathbf{0}}
\newcommand{\one}{\mathbf{1}}
\newcommand{\two}{\mathbf{2}}

\newcommand{\Var}{\mathrm{V}}

\newcommand{\Def}[1]{\textit{\textbf{#1}}}

\newcommand{\HA}{\mathsf{HA}}

\newcommand{\KP}{\mathsf{KP}}
\newcommand{\Sc}{\mathsf{Sc}}
\newcommand{\Md}{\mathsf{Md}}

\newcommand{\QQ}{\mathsf{Q}}

\newcommand{\QVar}{\mathcal Q}

\newcommand{\bydef}{:=}

\newcommand{\pfltr}[1]{[#1)}
\newcommand{\pidl}[1]{(#1]}

\newcommand{\apair}[2]{\langle #1; #2 \rangle}
\newcommand{\set}[2]{\{ #1 \ : \ #2 \}}

\newcommand{\inc}{||} 

\newcommand{\LogL}{\mathsf{L}} 
\newcommand{\Int}{\mathsf{Int}}

\newcommand{\An}[1]{\alg A^{(#1)}}

\DeclareUnicodeCharacter{8658}{$\Rightarrow$}

\newcommand{\PAlg}{\alg P}

\journal{Annals of pure and Applied Logic}
\usepackage{lineno}
\begin{document}

\begin{frontmatter}



\title{Hereditarily Structurally Complete Superintuitionistic Logics and Primitive Varieties of Heyting Algebras}


\author{Alex Citkin} 

\affiliation{organization={Metropolitan Telecommunications, New York, USA},
            addressline={55 Water Str.}, 
            city={New York},
            postcode={10041}, 
            state={NY},
            country={USA}}

\begin{abstract}
We give an algebraic proof of the criterion for hereditary structural completeness of an intermediate logic, or, equivalently, of the primitiveness of a variety of Heyting algebras.
\end{abstract}


\end{frontmatter}

\section{Introduction.}

By a \Def{logic} we understand a set of propositional formulas built in a regular way from propositional variables and the connectives $\land, \lor, \to$, and $\neg$, and which is closed under the rules of Modus Ponens and simultaneous substitution. Logic $\LogL$ is called \Def{superintuitionistic} (or \Def{si-logic}, for short) if it contains all the theorems of Intuitionistic Propositional Logic $\Int$ (see, e.g., \cite{Kleene_Intro}).

Heyting (or pseudo-Boolean \cite{Rasiowa_Sikorski}) algebras are usually used as the algebraic semantics for si-logics. A Heyting algebra is an algebra $\alg A = \apair{A}{\land, \lor, \to, \neg}$, where $\apair{A}{\land, \lor}$ is a distributive lattice and $\to$ and $\neg$ are the respective relative pseudocomplement and pseudocomplement. Every Heyting algebra contains a top element, denoted by $\one$, and a formula $\alpha$ is considered to be \Def{valid} in $\alg A$ if $\nu(\alpha) = \one$ for any valuation $\nu$ in $\alg A$. A Heyting algebra $\alg A$ is called a \Def{model} of an si-logic $\LogL$ if every formula from $\LogL$ is valid in $\alg A$.

A (finitary structural inference) \Def{rule} is an ordered pair $(A; \alpha)$ (also written as $A/\alpha$), where the first component $A$ is a finite (possibly empty) set of formulas, and the second component $\alpha$ is a formula. Given a logic $\LogL$, a rule $A/\alpha$ is said to be \Def{admissible in} $\LogL$ if for every substitution $\sigma$ such that $\sigma(\alpha) \notin \LogL$, there exists a formula $\beta \in A$ with $\sigma(\beta) \notin \LogL$. For example, the rule 
\begin{align*}
	\neg p \to (q \lor r) / ((\neg p \to q) \lor (\neg p \to r))
\end{align*}
is admissible in every si-logic (see \cite[Theorem 1]{Prucnal_Two_1979}). A rule $\alpha_0, \dots, \alpha_n / \alpha$ is said to be \Def{derivable in} a logic $\LogL$ if the formula $((\alpha_0 \land \dots \land \alpha_n) \to \alpha)$ belongs to $\LogL$. For instance, the aforementioned rule is derivable in classical logic but not derivable in intuitionistic logic.

In \cite{Pogorzelski_Structural_1971}, Pogorzelski introduced the notion of \Def{structural completeness}: a logic is structurally complete if every rule admissible in it is derivable. For instance, classical propositional logic is structurally complete, while intuitionistic propositional logic is not.
Soon after structural completeness had been introduced, Dzik and Wroński observed in \cite{Dzik_Wronski_1973} that any si-logic extending Dummett’s logic (the smallest si-logic containing the formula $((p \to q) \lor (q \to p))$) is structurally complete. This observation led the author to introduce (see \cite{Citkin_1978}) the notion of \Def{hereditary structural completeness}: a logic is hereditarily structurally complete if it and all its extensions are structurally complete. For example, Dummett’s logic is hereditarily structurally complete. Moreover, the following simple criterion of hereditary structural completeness has been established in \cite{Citkin_1978}:

\begin{theorem}[\cite{Citkin_1978}]\label{th-mainl}
	An si-logic $\LogL$ is hereditarily structurally complete if and only if none of the prohibited algebras $\PAlg_i$, $i \in [1,5]$ (whose Hasse diagrams are depicted in Fig.~\ref{fig-alg}) is a model of $\LogL$.
\end{theorem}

\begin{figure}[ht]
	\begin{center}
		\begin{tabular}{ccccc}
			\begin{tikzpicture}[scale= 0.8]
				\draw[fill] (0.5,1.5) -- (0.5,2);	
				\draw[fill] (0,0) -- (-0.5,0.5);	
				\draw[fill] (0,0) -- (1,1);	
				\draw[fill] (0,1) -- (0.5,0.5);	
				\draw[fill] (0.5,1.5) -- (-0.5,0.5);	
				\draw[fill] (0.5,1.5) -- (1,1);	
				\draw[fill] (0,0) circle [radius=0.07];
				\draw[fill] (-0.5,0.5) circle [radius=0.07];
				\draw[fill] (0.5,0.5) circle [radius=0.07];
				\draw[fill] (0,1) circle [radius=0.07];				
				\draw[fill] (0.5,2) circle [radius=0.07];				
				\draw[fill] (1,1) circle [radius=0.07];				
				\draw[fill] (0.5,1.5) circle [radius=0.07];			
				\draw[fill,white] (0,-0.5) circle [radius=0.07];			
				\node[left]  at (1,2) {\footnotesize $\quad$};	
			\end{tikzpicture}	
			&	
			\begin{tikzpicture}[scale=0.9]
				\draw[fill] (0.5,1.5) -- (0.5,2);	
				\draw[fill] (0,0) -- (-0.5,0.5);	
				\draw[fill] (0,0) -- (1,1);	
				\draw[fill] (0,1) -- (0.5,0.5);	
				\draw[fill] (0.5,1.5) -- (-0.5,0.5);	
				\draw[fill] (0.5,1.5) -- (1,1);	
				\draw[fill] (0,0) -- (0,-0.5);	
				\draw[fill] (0,0) circle [radius=0.07];
				\draw[fill] (-0.5,0.5) circle [radius=0.07];
				\draw[fill] (0.5,0.5) circle [radius=0.07];
				\draw[fill] (0,1) circle [radius=0.07];				
				\draw[fill] (0.5,2) circle [radius=0.07];				
				\draw[fill] (1,1) circle [radius=0.07];				
				\draw[fill] (0.5,1.5) circle [radius=0.07];			
				\draw[fill] (0,-0.5) circle [radius=0.07];			
				\node[right]  at (1,2) {\footnotesize $\quad$};	
			\end{tikzpicture}	
			&		
			\begin{tikzpicture}[scale=0.9]
				\draw[fill] (0,1) -- (0,1.5);	
				\draw[fill] (-0.5,0.5) -- (0,1);	
				\draw[fill] (0.5,0.5) -- (0,1);	
				\draw[fill] (-0.5,0.5) -- (0,0);	
				\draw[fill] (0.5,0.5) -- (0,0);	
				\draw[fill] (-0.5,0) -- (0,0.5);	
				\draw[fill] (0.5,0) -- (0,0.5);	
				\draw[fill] (-0.5,0) -- (-0.5,0.5);	
				\draw[fill] (0.5,0) -- (0,-0.5);	
				\draw[fill] (-0.5,0) -- (0,-0.5);	
				\draw[fill] (0,0.5) -- (0,1);	
				\draw[fill] (0,-0.5) -- (0,0);	
				\draw[fill] (0.5,0) -- (0.5,0.5);	
				\draw[fill] (0,0) -- (0,-0.5);	
				\draw[fill] (0,0) circle [radius=0.07];
				\draw[fill] (0,1) circle [radius=0.07];
				\draw[fill] (0,1.5) circle [radius=0.07];
				\draw[fill] (-0.5,0.5) circle [radius=0.07];
				\draw[fill] (0.5,0.5) circle [radius=0.07];
				\draw[fill] (0,0.5) circle [radius=0.07];	
				\draw[fill] (-0.5,0) circle [radius=0.07];	
				\draw[fill] (0.5,0) circle [radius=0.07];	
				\draw[fill] (0,-0.5) circle [radius=0.07];	
				\draw[fill,white] (0,-1) circle [radius=0.07];	
				\node[right]  at (1,0) {\footnotesize $ $};
			\end{tikzpicture}	
			&	
			\begin{tikzpicture}[scale=0.9]
				\draw[fill] (0,1) -- (0,1.5);	
				\draw[fill] (-0.5,0.5) -- (0,1);	
				\draw[fill] (0.5,0.5) -- (0,1);	
				\draw[fill] (-0.5,0.5) -- (0,0);	
				\draw[fill] (0.5,0.5) -- (0,0);	
				\draw[fill] (-0.5,0) -- (0,0.5);	
				\draw[fill] (0.5,0) -- (0,0.5);	
				\draw[fill] (-0.5,0) -- (-0.5,0.5);	
				\draw[fill] (0.5,0) -- (0,-0.5);	
				\draw[fill] (-0.5,0) -- (0,-0.5);	
				\draw[fill] (0,0.5) -- (0,1);	
				\draw[fill] (0,-0.5) -- (0,0);	
				\draw[fill] (0.5,0) -- (0.5,0.5);	
				\draw[fill] (0,-1) -- (0,0);	
				\draw[fill] (0,0) circle [radius=0.07];
				\draw[fill] (0,1) circle [radius=0.07];
				\draw[fill] (0,1.5) circle [radius=0.07];
				\draw[fill] (-0.5,0.5) circle [radius=0.07];
				\draw[fill] (0.5,0.5) circle [radius=0.07];
				\draw[fill] (0,0.5) circle [radius=0.07];	
				\draw[fill] (-0.5,0) circle [radius=0.07];	
				\draw[fill] (0.5,0) circle [radius=0.07];	
				\draw[fill] (0,-0.5) circle [radius=0.07];	
				\draw[fill] (0,-1) circle [radius=0.07];	
			\end{tikzpicture}
			&
			\begin{tikzpicture}[scale=0.9]
				\draw[fill] (0,2) -- (0,2.5);	
				\draw[fill] (-0.5,0.5) -- (0.5,1.5);	
				\draw[fill] (0.5,0.5) -- (-0.5,1.5);	
				\draw[fill] (-0.5,0.5) -- (0,0);	
				\draw[fill] (0.5,0.5) -- (0,0);	
				\draw[fill] (-0.5,0) -- (0,0.5);	
				\draw[fill] (0.5,0) -- (0,0.5);	
				\draw[fill] (-0.5,0) -- (-0.5,0.5);	
				\draw[fill] (0.5,0) -- (0,-0.5);	
				\draw[fill] (-0.5,0) -- (0,-0.5);	
				\draw[fill] (0,0.5) -- (0,1);	
				\draw[fill] (0,0) -- (1,1);	
				\draw[fill] (0,0) -- (-1,1);	
				\draw[fill] (0.5,0) -- (0.5,0.5);	
				\draw[fill] (0,2) -- (1,1);	
				\draw[fill] (0,2) -- (-1,1);	
				\draw[fill] (0,0) -- (0,-0.5);	
				\draw[fill] (0,0) circle [radius=0.07];
				\draw[fill] (0,1) circle [radius=0.07];
				\draw[fill] (0,2) circle [radius=0.07];
				\draw[fill] (-0.5,0.5) circle [radius=0.07];
				\draw[fill] (0.5,0.5) circle [radius=0.07];
				\draw[fill] (0,0.5) circle [radius=0.07];	
				\draw[fill] (-0.5,0) circle [radius=0.07];	
				\draw[fill] (0.5,0) circle [radius=0.07];	
				\draw[fill] (0,-0.5) circle [radius=0.07];	
				\draw[fill,white] (0,-1) circle [radius=0.07];	
				\draw[fill] (1,1) circle [radius=0.07];			\draw[fill] (-1,1) circle [radius=0.07];		\draw[fill] (-0.5,1.5) circle [radius=0.07];
				\draw[fill] (0.5,1.5) circle [radius=0.07];
				\draw[fill] (0,2.5) circle [radius=0.07];
				\node[right]  at (1,0) {\footnotesize $ $};
			\end{tikzpicture}\\
			$\PAlg_1$ &\!\!\!\!\!\!\!\!\!$\PAlg_2$ &\!\!\!\!$\PAlg_3$ &\ \ $\PAlg_4$ & $\PAlg_5$
		\end{tabular} 
	\end{center}
	\caption{Prohibited Algebras.}\label{fig-alg}	
\end{figure}
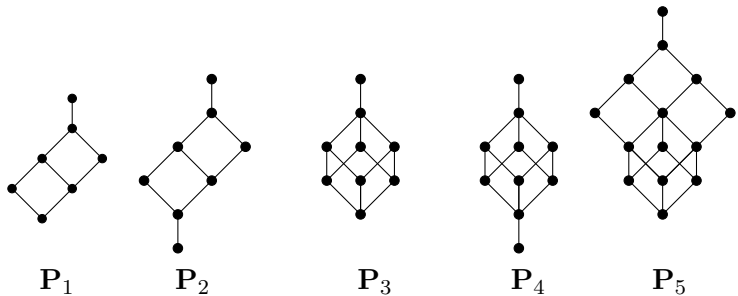
\begin{note}
	Throughout this paper, the notations $\PAlg_1$ – $\PAlg_5$ will always refer to the five algebras depicted in Fig.~\ref{fig-alg} which we call \Def{prohibited}.
\end{note}

We immediately obtain the following result from Theorem~\ref{th-mainl}.

\begin{theorem}
	There exists an algorithm which, given a finite set of formulas $A$, determines whether the si-logic containing all these formulas is hereditarily structurally complete.
\end{theorem}

Furthermore:

\begin{theorem}\label{th-fa}
	Every hereditarily structurally complete si-logic is finitely axiomatizable.
\end{theorem}

The proof of Theorem~\ref{th-mainl} is based on the following observation, which is important in its own right.

\begin{theorem} \label{th-lfl}
	Any si-logic for which prohibited algebra  $\PAlg_2$ is not a model is locally tabular.
\end{theorem}

In \cite{Citkin_1978}, it was also observed that not every structurally complete si-logic is hereditarily structurally complete: $\PAlg_4$ is a model of Medvedev’s logic, but, as shown in \cite{Prucnal_Structural_1976}, this logic is structurally complete. More examples of structurally complete extensions of Medvedev’s logic that are not hereditarily structurally complete can be found in \cite{Skvortsov_Prucnal_1998}.

A proof of these theorems was published in 1987 in Russian in \cite{Citkin1987} and, unfortunately, it has never been translated into English; consequently, it remains largely unknown (see, e.g., \cite{Bergman_Structural_1991}). Alternative proofs have been published in 1997 in \cite[Theorem~54.8]{Rybakov_Book} and in 2022 in \cite{Bezhanishvili_Moraschini_Citkins_2022}. In this paper, we present the original (slightly revised) algebraic proofs of these theorems. Using the same approach, hereditarily structurally complete modal logics were also characterized (see \cite{Rybakov_Book}).

We employ algebraic semantics to establish the theorems. There exists a one-to-one correspondence between si-logics and varieties of Heyting algebras, and a logic is hereditarily structurally complete if and only if the corresponding variety of Heyting algebras is \Def{primitive} (or \emph{deductive} in the terminology of \cite{Bergman_Structural_1991}). Recall (see, e.g., \cite{GorbunovBookE}) that a variety is primitive if every subquasivariety of it is a subvariety. Consequently, we obtain the following results.


\begin{theorem}\label{th-main}
	A variety of Heyting algebras is primitive if and only if it omits all prohibited algebras.
\end{theorem}

\begin{theorem}\label{th-fb}
	All primitive varieties of Heyting algebras are finitely based.
\end{theorem}

\begin{theorem}\label{th-lfa}
	Any variety of Heyting algebras omitting prohibited algebra $\PAlg_2$ is locally finite.
\end{theorem}

The remainder of this paper is structured as follows. The paper is self-contained, and we provide all necessary definitions and background on Heyting algebras and their varieties in Sections~\ref{sec-prel}--\ref{sec-pdpr}. Readers familiar with Heyting algebras and their properties may skip these sections, except for Theorem~\ref{th-proj}, which plays a crucial role in the proof of the main result. In Section~\ref{sec-qo}, we examine the properties of quasi-orders in classes of finite subdirectly irreducible algebras. In particular, we show that the class of finite projective Heyting algebras is well quasi-ordered with respect to the subalgebra relation. Finally, we prove Theorem~\ref{th-main} in Section~\ref{sec-main} and Theorem~\ref{th-fb} in Section~\ref{sec-fb}.

\section{Preliminaries.}\label{sec-prel}

\subsection{Heyting Algebras: Definitions and Properties.}

Any Heyting algebra $\alg A$ contains the greatest and least elements, denoted by $\one_\alg A$ and $\zero_\alg A$, respectively (we omit the reference to the algebra when no confusion arises). For $a \in \alg A$, we say that $a$ is \Def{regular} if $\neg\neg a = a$, and that $a$ is \Def{dense} if $\neg a = \zero$; otherwise, $a$ is called an \Def{ordinary} element. If $a, b \in \alg A$ and either $a \leq b$ or $b \leq a$, then $a$ and $b$ are \Def{comparable}; otherwise, they are \Def{incomparable}, which we denote by $a \inc b$. By $\HA$ we denote the variety of all Heyting algebras. We write $\alg A \leq \alg B$ to mean that $\alg A$ is (isomorphic to) a subalgebra of $\alg B$, and $\alg A \inc \alg B$ to mean that $\alg A \nleq \alg B$ and $\alg B \nleq \alg A$.

\paragraph{Coatoms.}

Let $\alg A$ be a Heyting algebra and $a, b \in \alg A$. Then $a$ \Def{covers} $b$ if $b \leq a$ and there are no elements strictly between $b$ and $a$. Elements covered by $\one$ are called \Def{coatoms}.

Observe that an ordinary element cannot be a coatom, because if $a$ is ordinary, then $a < \neg\neg a < \one$.

\begin{proposition}\label{pr-coatoms}
	Let $\alg A$ be a Heyting algebra, and let $a, b \in \alg A$ be coatoms. Then $a \to b = b$. 
\end{proposition}

The proof follows immediately from the definition of relative pseudo-complementation.

\begin{proposition}\label{pr-coatZ7}
	Let $\alg A$ be a Heyting algebra with coatoms $a$ and $b$, where $a$ is dense and $b$ is regular. Then these elements generate a subalgebra whose Hasse diagram is depicted in Fig.~\ref{fig-coat}.
	\begin{figure}[h]
		\centering
		\begin{tikzpicture}[scale=0.6]
			\draw[fill] (0,0) -- (-0.5,0.5);	
			\draw[fill] (0,0) -- (1,1);	
			\draw[fill] (0,1) -- (0.5,0.5);	
			\draw[fill] (0.5,1.5) -- (-0.5,0.5);	
			\draw[fill] (1,1) -- (0.5,1.5);	
			\draw[fill] (0,0) circle [radius=0.07];
			\draw[fill] (-0.5,0.5) circle [radius=0.07];
			\draw[fill] (0.5,0.5) circle [radius=0.07];
			\draw[fill] (0,1) circle [radius=0.07];				
			\draw[fill] (0.5,1.5) circle [radius=0.07];
			\draw[fill] (1,1) circle [radius=0.07];
			\node[right]  at (0.5,0.5) {\footnotesize $b \land a$};
			\node[right]  at (1,1) {\footnotesize $b$};
			\node[above]  at (0.5,1.5) {\footnotesize $\one$};
			\node[below]  at (0,0) {\footnotesize $\zero$};
			\node[left]   at (-0.5,0.5) {\footnotesize $\neg b$};
			\node[left]   at (0,1) {\footnotesize $a$};
		\end{tikzpicture}
		\caption{Subalgebra generated by the coatoms.}\label{fig-coat}
	\end{figure}
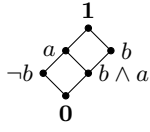
	\begin{proof}
		First, we show that $\neg b \leq a$. Indeed, since $a$ is a coatom, if $\neg b \not\leq a$, then $a \lor \neg b = \one$. On the other hand, as $a$ and $b$ are coatoms and therefore incomparable, we have $a \lor b = \one$. Hence,
		\[
		\one = (a \lor b) \land (a \lor \neg b) = a \lor (b \land \neg b) = a,
		\]
		contradicting that $a$ is a coatom.
		
		Next, by Proposition~\ref{pr-coatoms}, we have $a \to b = b$ and $b \to a = a$. In addition, $b \to (b \land a) = (b \to b) \land (b \to a) = a$.
		
		Since $b$ is a coatom and regular, $\neg b$ is incomparable with $b$; hence $b < b \lor \neg b$, which implies $b \lor \neg b = \one$. As established above, $\neg b \leq a$; therefore, $a \lor b = \one$.
		
		Finally, since $a$ is dense, we have $a \to \zero = \zero$. As $b$ is regular, $b \to \neg b = \neg b$ and $\neg b \to b = b$. Therefore,
		\begin{align*}
			a \to \neg b &= a \to (b \to \zero) = b \to (a \to \zero) = b \to \zero = \neg b,\\
			(b \land a) \to \neg b &= (b \to \neg b) \land (a \to \neg b) = \neg b \land \neg b = \neg b,\\
			\neg b \to (b \land a) &= (\neg b \to b) \land (\neg b \to a) = (\neg b \to b) = b,\\
			\neg(b \land a) &= (b \land a) \to \zero = b \to (a \to \zero) = b \to \zero = \neg b.
		\end{align*}
	\end{proof}
\end{proposition}

\paragraph{Ordinary elements.}

We will often refer to the following simple observation.

\begin{proposition}\label{pr-dr}
	Let $\alg A$ be a Heyting algebra and $a,b \in \alg A$. 
	If $a$ is dense, $b$ is regular, and $a \inc b$, then $a \land b$ is an ordinary element.
\end{proposition}

\begin{proof}
	Since $a \inc b$, both elements are distinct from $\one$. 
	Hence, $a$ is dense but not regular, and $b$ is regular but not dense.
	
	Let $c = a \land b$. We aim to show that $c$ is an ordinary element.
	
	First, $c$ is not dense, because otherwise, since $c \leq b$, we would have $b$ dense as well.
	
	Next, we show that $c$ is not regular.
	
	For contradiction, assume that $c$ is regular. Then, since $a$ is dense and $b$ is regular,
	\begin{align*}
		c = \neg\neg c = \neg\neg(a \land b) = \neg(\neg a \lor \neg b) = \neg\neg b = b,
	\end{align*}
	that is, $a \land b = b$, and consequently $b \leq a$, contradicting $a \inc b$.
\end{proof}

\begin{corollary}\label{cor-dnr}
	Let $\alg A$ be a Heyting algebra and $a,b \in \alg A$. 
	If $a$ is dense, $b$ is not dense, and $a \inc b$, then $a \land \neg\neg b$ is an ordinary element.
\end{corollary}

\begin{proof}
	Since $\neg\neg b$ is a regular element, by Proposition~\ref{pr-dr}, it suffices to show that $a \inc \neg\neg b$. 
	
	Indeed, $a \leq \neg\neg b$ is impossible, because in this case $\neg\neg b$ would be dense and therefore $\neg\neg b = \one$, that is, $b$ would be dense—contrary to the assumption.  
	
	On the other hand, $\neg\neg b \leq a$ is also impossible, since this would imply 
	$b \leq \neg\neg b \leq a$, contradicting the assumption that $a \inc b$.
\end{proof}

If $\alg A$ is a Heyting algebra and $a \in \alg A$, we denote by $\pfltr{a}$ the principal filter generated by $a$, that is, $\pfltr{a} = \{b \in \alg A \mid a \leq b\}$, and by $\pidl{a}$ the principal ideal generated by $a$, that is, $\pidl{a} = \{b \in \alg A \mid b \leq a\}$. Each principal filter, viewed as a lattice, forms a Heyting algebra, denoted by $\alg A\pfltr{a}$. Note that if $a \neq \zero$, the algebra $\alg A\pfltr{a}$ is not a subalgebra of $\alg A$. Furthermore, $\pidl{a}$, viewed as a lattice, forms a Heyting algebra isomorphic to $\alg A / \pfltr{a}$.

\paragraph{Strong partial ordering.}

We now prove some properties of Heyting algebras that will be used later.

If $a,b \in \alg A \in \HA$, we say that $a$ is \Def{strongly greater than} $b$ (denoted $a \gg b$), or $b$ is \Def{strongly smaller} than $a$ (in symbols $b \ll a$), if $a \ge b$ and $a \to b = b$. Note that $\one \gg a$ for every $a \in \alg A$, and if $b \neq \one$ and $b \gg a$, then $b > a$. 

\begin{proposition}\label{pr_strgr}
	Let $\alg A$ be a Heyting algebra. Then, for any $a, b, c \in \alg A$,
	\begin{align}
		&\text{if } a \ll b \text{ and } b \leq c, \text{ then } a \ll c; \label{eq-strl}\\
		&\text{if } a \leq b \text{ and } b \ll c, \text{ then } a \ll c; \label{eq-strl1}\\		
		&\text{if } a \gg b \text{ and } b \ge c, \text{ then } a \gg c; \label{eq-strg}\\
		&\text{if } a \ge b \text{ and } b \gg c, \text{ then } a \gg c. \label{eq-strg1}		
	\end{align}
\end{proposition}

\begin{proof}
	It clearly suffices to prove \eqref{eq-strl} and \eqref{eq-strl1}. Observe that in both cases we have $a \leq c$, and it remains to show that $c \to a = a$. Note that $a \leq c$ implies $a \land c = a$, and by the definition of relative pseudocomplementation, it suffices to show that for any $d \in \alg A$ such that $c \land d \leq a$, we have $d \leq a$. 
	
	For \eqref{eq-strl}, assume that $c \land d \leq a$. Since $b \leq c$, by monotonicity we obtain $b \land d \leq a$. Then, by the definition of pseudocomplementation, $d \leq b \to a$. By assumption, $a \ll b$, which implies $b \to a = a$, and hence $d \leq a$, as desired. 
	
	The case \eqref{eq-strl1} can be proved in a similar way.
\end{proof}

\begin{corollary}
	For any $\alg A \in \HA$ and $a,b \in \alg A$,
	\begin{align}
		(b \lor (b \to a)) \gg a. \label{eq-node}
	\end{align}
\end{corollary}
\begin{proof}
	Indeed,
	\begin{align*}
		&a \leq b \to a \leq (b \lor (b \to a)), \text{ and}\\
		&(b \lor (b \to a)) \to a = (b \to a) \land ((b \to a) \to a) \leq a.
	\end{align*}
\end{proof}

\begin{proposition}\label{pr-fourpr}
	Let $\alg A$ be a Heyting algebra and $a, b \in \alg A$. Then the elements $a$, $b \to a$, $(b \to a) \to a$, and $(b \to a) \lor ((b \to a) \to a)$ form a relatively pseudocomplemented sublattice: 
	\begin{figure}[ht]
		\begin{center}
			\begin{tikzpicture}[scale= 1] 
				\draw[fill] (0,0) -- (-0.5,0.5);	
				\draw[fill] (0,0) -- (0.5,0.5);	
				\draw[fill] (0,1) -- (0.5,0.5);	
				\draw[fill] (0,1) -- (-0.5,0.5);	
				\draw[fill] (0,0) circle [radius=0.07];
				\draw[fill] (-0.5,0.5) circle [radius=0.07];
				\draw[fill] (0.5,0.5) circle [radius=0.07];
				\draw[fill] (0,1) circle [radius=0.07];				
				\node[left]  at (-0.5,0.5) {\footnotesize $b \to a$};
				\node[right]  at (0.5,0.5) {\footnotesize $(b \to a) \to a$};
				\node[above]  at (0,1) {\footnotesize $(b \to a) \lor ((b \to a) \to a$)};
				\node[below]  at (0,0) {\footnotesize $a$};
				
			\end{tikzpicture}
		\end{center}
		\caption{Subalgebra}
		\label{fig-pr-4}
	\end{figure}
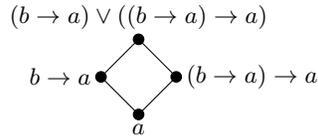
\end{proposition}

\begin{proof}
	To confirm that these elements form a lattice, we must verify that $(b \to a) \land ((b \to a) \to a) = a$; the rest follows easily.
	
	Indeed, $a \leq (b \to a)$ and $a \leq ((b \to a) \to a)$; hence $a \leq (b \to a) \land ((b \to a) \to a)$. Conversely, $(b \to a) \land ((b \to a) \to a) \leq a$.
	
	Next, we confirm that this lattice is relatively pseudocomplemented by showing that $((b \to a) \to a) \to a = (b \to a)$.
	
	Indeed, since $b \land (b \to a) \leq a$, we have $b \leq (b \to a) \to a$. Therefore,
	\[
	((b \to a) \to a) \land (((b \to a) \to a) \to a) \leq a
	\]
	implies
	\[
	b \land (((b \to a) \to a) \to a) \leq a,
	\]
	that is, $(((b \to a) \to a) \to a) \leq (b \to a)$.
	
	Furthermore, since $(b \to a) \land ((b \to a) \to a) \leq a$, it follows that $b \to a \leq ((b \to a) \to a) \to a$. Hence,
	\begin{align}
		((b \to a) \to a) \to a = (b \to a). \label{eq-proj-1}
	\end{align}
	
	Finally, we show that $((b \to a) \to a) \to (b \to a) = (b \to a)$. Indeed,
	\[
	((b \to a) \to a) \to (b \to a) = b \to (((b \to a) \to a) \to a),
	\]
	and from~\eqref{eq-proj-1} we know that $((b \to a) \to a) \to a = (b \to a)$; hence,
	\[
	((b \to a) \to a) \to (b \to a) = b \to (b \to a) = b \to a.
	\]
	
\end{proof}

\subsection{Nodes.}

Let $\alg A$ be an algebra. An element $a \in \alg A$ is said to be a \Def{node} if $a$ is comparable with every element of $\alg A$. Thus, $\zero$ and $\one$ are always \Def{trivial} nodes. We say that a nontrivial algebra is \Def{nodeless} if it does not contain nontrivial nodes.

\begin{proposition}\label{pr-node}
	Let $\alg A$ be a Heyting algebra with a nontrivial node $a \in \alg A$. Then, $ b \ll a$ holds for any $b < a$.
\end{proposition}
\begin{proof}
	Assume that $b < a$. Hence to prove $b \ll a$ we need only to show that $a \to b = b$. Moreover, since trivially $a \land b \leq b$, we only need to show that for any $c \in \alg A$,  $a \land c \leq b$, implies $c \leq b$. 
	
	Let $a \in \alg A$ and $a \land c \leq b$. Because $a$ is a node, either $c \leq a$, or $a \leq c$.
	
	If $c \leq a$, then $c = a \land c$ and by assumption, $c = a \land c \leq b$, as desired.
	
	If $a \leq c$, then we have $b < a \leq c$, as desired.
\end{proof}

By Propositions \ref{pr_strgr} and \ref{pr-node} we obtain the following corollaries.

\begin{corollary}\label{cor-nodesep}
	Let $\alg A$ be a Heyting algebra with a nontrivial node $a \in \alg A$. Then, for any two elements $b,c \in \alg A$, $b < a < c$ implies $b \ll c$.
\end{corollary}

\begin{corollary}\label{cor-nodesimpl}
	Let $\alg A$ be a Heyting algebra, and let $a, b \in \alg A$ be two nodes with $a < b$. Then $a \ll b$.
\end{corollary}

\begin{proposition}\label{pr-nodesubalg}
	Let $a_0, a_1$ be nodes of an algebra $\alg A$ with $a_0 \leq a_1$. Then the set 
	\[
	A' \bydef \pidl{a_0} \cup \pfltr{a_1}
	\]
	forms a subalgebra.
\end{proposition}	
\begin{proof}
	First, note that $A'$ is closed under $\land$ and $\lor$, and contains $\one$ and $\zero$. Therefore, we only need to verify that $A'$ is closed under $\to$.  
	
	Let $b,c \in A'$; we need to show that $b \to c \in A'$. We start with the obvious case when $b \leq c$, in which case $b \to c = \one \in A'$. So, we assume that $b \not\leq c$.
	
	Since $a_0$ is a node, either (1) $c < a_0$, or (2) $c = a_0$, or (3) $c > a_0$. 
	
	\textbf{Case (1).} Suppose $c < a_0$. Again, as $a_0$ is a node, either (i) $b \leq a_0$, or (ii) $a_0 < b$.
	
	(i) Suppose that $b \leq a_0$. Then for every $d \in A'$, it is impossible that $b \land d \leq c$ when $d \ge a_0$. Indeed, in this case we have $b \leq a_0 \leq d$, and hence $b \land d = b$, while by our assumption $b \not\leq c$. Thus, since $b \land (b \to c) \leq c$ and $a_0$ is a node, we have $b \to c \leq a_0$, i.e. $b \to c \in \pidl{a_0} \subseteq A'$.
	
	(ii) If $c < a_0 < b$, then by Corollary~\ref{cor-nodesimpl}, $c \ll b$, and consequently $b \to c = c \in A'$.
	
	\textbf{Case (2).} Suppose that $c = a_0$. Since $b \not\leq c$, it follows that $a_1 \leq b$. There are two subcases to consider: (i) $a_0 = a_1$ and (ii) $a_0 < a_1$.
	
	(i) Suppose that $a_0 = a_1$. Then $a_1 = c$, and we have $a_1 = c \leq (b \to c)$, i.e. $b \to c \in \pfltr{a_1} \subseteq A'$.
	
	(ii) If $c = a_0 < a_1$, then by Proposition~\ref{pr-node}, $c \ll a_1$. Since $b \not\leq c$, we have $a_1 \leq b$, and by~\eqref{eq-strl}, $c \ll b$. Thus, $b \to c = c \in A'$.
	
	\textbf{Case (3).} If $c > a_0$, then $c \ge a_1$, and hence $a_1 \leq c \leq b \to c$. Therefore, $b \to c \in \pfltr{a_1} \subseteq A'$.
\end{proof}

\begin{corollary}\label{cr-onenode}
	If $\alg A$ is a Heyting algebra with a nontrivial node $a$, then $\pidl{a} \cup \{\one\}$ forms a subalgebra. 
\end{corollary}

\begin{corollary}\label{cr-threenodes}
	Let $\alg A$ be a Heyting algebra generated by a set of elements $G$, and let $a,c$ be nodes and $b$ be an element such that $a < b < c$. Then there exists a generator $g \in G$ such that $a < g < c$.
\end{corollary}
\begin{proof}
	Indeed, by Proposition~\ref{pr-nodesubalg}, $\pidl{a} \cup \pfltr{c}$ forms a subalgebra, and $b \notin \pidl{a} \cup \pfltr{c}$ means that $\pidl{a} \cup \pfltr{c}$ is a proper subalgebra. Therefore, there exists a generator $g \in G$ such that $g \notin \pidl{a} \cup \pfltr{c}$, and because $a$ and $c$ are nodes, we have $a < g < c$.  
\end{proof}

\begin{corollary}\label{cr-ngenfinnodes}
	Any finitely generated Heyting algebra contains only finitely many nodes.
\end{corollary}

The following corollary provides a sharper bound on the number of nodes.

\begin{corollary} \label{cr-ngen2nnodes}
	Any $n$-generated Heyting algebra $\alg A$ has at most $2n + 2$ nodes.
\end{corollary}
\begin{proof}
	Let $\alg A$ be an algebra generated by elements $g_0, \dots, g_{n-1}$. By Corollary~\ref{cr-ngenfinnodes}, the algebra $\alg A$ has only finitely many nodes. Assume, for the sake of contradiction, that $\alg A$ contains more than $2n + 2$ nodes. Let
	\[
	a_0 < a_1 < \dots < a_{2n} < a_{2n+1} < a_{2n+2}
	\]
	be nodes. Then, by Corollary~\ref{cr-threenodes}, for every $j \leq n + 1$ there exists a generator $g_s$ such that $a_{2j} < g_s < a_{2j+2}$, which is impossible.
\end{proof}

\subsection{Coalesced Ordinal Sums.} 

The concept of a coalesced sum is an important tool in the study of finitely generated algebras. The notion of the sum of Heyting algebras was introduced in \cite{Troelstra_Intermediate_1965}, and it has also appeared under various names in the literature: the \emph{ordinal coalesced sum} \cite{Galatos_et_Book}, the \emph{Troelstra sum} \cite{Skura_Decision_1991}, the \emph{sequential sum} \cite{Kuznetsov_Gerchiu}, the \emph{star sum} \cite{Balbes_Horn_1970}, the \emph{horizontal sum} \cite{Day_1973}, and the \emph{vertical sum} \cite{Bezhanishvili_N_PhD}. Coalesced ordinal sums have been extensively employed in these and many other works.

Let $\alg A = \apair{A}{\land,\lor,\to,\one}$ and $\alg B = \apair{B}{\land,\lor,\to,\one}$ be Heyting algebras. A \Def{coalesced ordinal sum} of $\alg A$ and $\alg B$ (or simply a \Def{sum}) is the algebra $\alg A + \alg B$ consisting of sublattices $\alg A'$ and $\alg B'$ isomorphic, respectively, to $\apair{A}{\land,\lor,\one}$ and $\apair{B}{\land,\lor,\one}$, such that $\alg A' \cap \alg B'$ contains exactly one common element — namely, the greatest element of $\alg A'$ and the least element of $\alg B'$. In other words, we identify $\one_\alg A$ with $\zero_\alg B$.

The Hasse diagram of $\alg A + \alg B$ can be obtained by placing the diagram of $\alg B$ above that of $\alg A$ and identifying the top element of $\alg A$ with the bottom element of $\alg B$.

\begin{proposition}\label{pr-sumsub}
	Let $\alg A, \alg B, \alg C, \alg C_1$, and $\alg C_2$ be finite nontrivial Heyting algebras. Then:
	\begin{itemize}
		\item[(a)] if $\alg A \leq \alg B$, then $\alg A + \alg C \leq \alg B + \alg C$;
		\item[(b)] if $\alg C_1 \leq \alg{C}_2$, then $\alg A + \alg C_1 \leq \alg A + \alg C_2$;
		\item[(c)] if $\alg A \inc \alg B$ and $\alg A + \alg C_1 \leq \alg A + \alg C_2$, then $\alg B + \alg C_1 \leq \alg B + \alg C_2$.
	\end{itemize}
\end{proposition}

\begin{proof}
	(a) It is straightforward to verify that if a map $f$ embeds $\alg A$ into $\alg B$, then the map
	\[
	f'(a) =
	\begin{cases}
		f(a), & \text{if } a \in \alg A,\\
		a, & \text{otherwise},
	\end{cases}
	\]
	embeds $\alg A + \alg C$ into $\alg B + \alg C$.
	
	(a) If a map $f$ embeds $\alg C_1$ into $\alg C_2$, then the map
	\[
	f'(a) =
	\begin{cases}
		a, & \text{if } a \in \alg A,\\
		f(a), & \text{otherwise},
	\end{cases}
	\]
	embeds $\alg A + \alg C_1$ into $\alg A + \alg C_2$.

	(c) Similarly, if a map $f$ embeds $\alg A + \alg C_1$ into $\alg A + \alg C_2$, then the map
	\[
	f'(a) =
	\begin{cases}
		a, & \text{if } a \in \alg B,\\
		f(a), & \text{otherwise},
	\end{cases}
	\]
	embeds $\alg B + \alg C_1$ into $\alg B + \alg C_2$.
\end{proof}

\begin{remark}
	Let us observe that in Proposition \ref{pr-sumsub} the converse of (b) is not true, and in (c) the requirement $\alg A\inc\alg B$ cannot be omitted: see Fig. \ref{fig-ce}.
	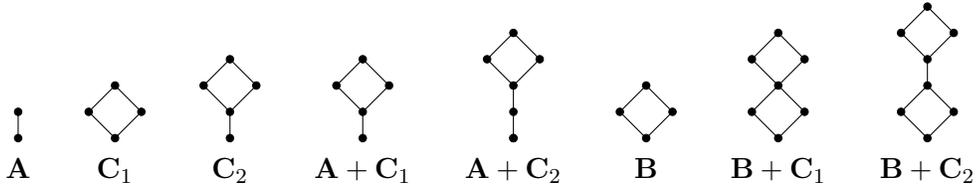
\begin{figure}[h!]
		\begin{center}
			\begin{tabular}{ccccccccccccccc}
				\begin{tikzpicture}[scale=0.7] 
					\draw[fill] (0,0) -- (0,0.5);	
					\draw[fill] (0,0) circle [radius=0.07];
					\draw[fill] (0,0.5) circle [radius=0.07]; 
				\end{tikzpicture} 
				&\!\!&
				\begin{tikzpicture}[scale=0.7] 
					\draw[draw] (0,0) -- (0.5,0.5) -- (0,1) --(-0.5,0.5)--(0,0);	
					\draw[fill] (0,0) circle [radius=0.07];
					\draw[fill] (0.5,0.5) circle [radius=0.07]; 
					\draw[fill] (-0.5,0.5) circle [radius=0.07]; 
					\draw[fill] (0,1) circle [radius=0.07]; 
				\end{tikzpicture}		
				&\!\!&
				\begin{tikzpicture}[scale=0.7] 
					\draw[draw] (0,0) -- (0.5,0.5) -- (0,1) --(-0.5,0.5)--(0,0)--(0,-0.5);	
					\draw[fill] (0,0) circle [radius=0.07];
					\draw[fill] (0.5,0.5) circle [radius=0.07]; 
					\draw[fill] (-0.5,0.5) circle [radius=0.07]; 
					\draw[fill] (0,1) circle [radius=0.07]; 
					\draw[fill] (0,-0.5) circle [radius=0.07];
				\end{tikzpicture} 
				&\!\!&
				\begin{tikzpicture}[scale=0.7] 
					\draw[draw] (0,0) -- (0.5,0.5) -- (0,1) --(-0.5,0.5)--(0,0)--(0,-0.5);	
					\draw[fill] (0,0) circle [radius=0.07];
					\draw[fill] (0.5,0.5) circle [radius=0.07]; 
					\draw[fill] (-0.5,0.5) circle [radius=0.07]; 
					\draw[fill] (0,1) circle [radius=0.07]; 
					\draw[fill] (0,-0.5) circle [radius=0.07];
				\end{tikzpicture} 
				&\!\!&
				\begin{tikzpicture}[scale=0.7] 
					\draw[draw] (0,0) -- (0.5,0.5) -- (0,1) --(-0.5,0.5)--(0,0)--(0,-1);	
					\draw[fill] (0,0) circle [radius=0.07];
					\draw[fill] (0.5,0.5) circle [radius=0.07]; 
					\draw[fill] (-0.5,0.5) circle [radius=0.07]; 
					\draw[fill] (0,1) circle [radius=0.07]; 
					\draw[fill] (0,-0.5) circle [radius=0.07];
					\draw[fill] (0,-1) circle [radius=0.07];
					\draw[fill] (0,-1) circle [radius=0.07];	
				\end{tikzpicture}
				&\!\!&
				\begin{tikzpicture}[scale=0.7] 
					\draw[draw] (0,0) -- (0.5,0.5) -- (0,1) --(-0.5,0.5)--(0,0);	
					\draw[fill] (0,0) circle [radius=0.07];
					\draw[fill] (0.5,0.5) circle [radius=0.07]; 
					\draw[fill] (-0.5,0.5) circle [radius=0.07]; 
					\draw[fill] (0,1) circle [radius=0.07]; 
				\end{tikzpicture}
				&\!\!&
				\begin{tikzpicture}[scale=0.7] 
					\draw[draw] (0,0) -- (0.5,0.5) -- (0,1) --(-0.5,0.5)--(0,0)--(0.5,-0.5)--(0,-1)--(-0.5,-0.5)--(0,0);	
					\draw[fill] (0,0) circle [radius=0.07];
					\draw[fill] (0.5,0.5) circle [radius=0.07]; 
					\draw[fill] (-0.5,0.5) circle [radius=0.07]; 
					\draw[fill] (0,1) circle [radius=0.07]; 
					\draw[fill] (0.5,-0.5) circle [radius=0.07];
					\draw[fill] (-0.5,-0.5) circle [radius=0.07];
					\draw[fill] (0,-1) circle [radius=0.07];
				\end{tikzpicture} 		
				&\!\!& 
				\begin{tikzpicture}[scale=0.7] 
					\draw[draw] (0,0) -- (0.5,0.5) -- (0,1) --(-0.5,0.5)--(0,0);
					\draw[draw] (0,1) -- (0,1.5)--(0.5,2)--(0,2.5)--(-0.5,2)--(0,1.5);	
					\draw[fill] (0,0) circle [radius=0.07];
					\draw[fill] (0.5,0.5) circle [radius=0.07]; 
					\draw[fill] (-0.5,0.5) circle [radius=0.07]; 
					\draw[fill] (0,1) circle [radius=0.07]; 
					\draw[fill] (0,1.5) circle [radius=0.07];
					\draw[fill] (0.5,2) circle [radius=0.07];
					\draw[fill] (-0.5,2) circle [radius=0.07];
					\draw[fill] (0,2.5) circle [radius=0.07];	
				\end{tikzpicture}				
				\\
				$\alg A$ &\!\!& $\alg C_1$ &\!\!& $\alg C_2$ &\!\!& $\alg A + \alg C_1$ &\!\!& $\alg A + \alg C_2$ &\!\!& $\alg B$ &\!\!& $\alg B + \alg C_1$ &\!\!& $\alg B + \alg C_2$
			\end{tabular}
		\end{center}		
		\label{fig-ce}
		\caption{Counterexamples.}
	\end{figure}	
\end{remark}

\begin{corollary}\label{cor_suminc}
	Let $\alg A, \alg B, \alg C, \alg C_1$, and $\alg C_2$ be finite nontrivial Heyting algebras such that $\alg A + \alg C_1 \inc \alg A + \alg C_2$. Then:
	\begin{itemize}
		\item[(a)] if $\alg A \leq \alg B$, then $\alg B + \alg C_1 \inc \alg B + \alg C_2$;
		\item[(b)] if $\alg A \inc \alg B$, then $\alg B + \alg C_1 \inc \alg B + \alg C_2$.
	\end{itemize}
\end{corollary}

If $\alg A$ is a Heyting algebra with a nontrivial node $a$, then $\alg A = \alg A\pidl{a} + \alg A\pfltr{a}$. Moreover, if $\alg A$ is a Heyting algebra and $\zero = a_0 < a_1 < \dots < a_m < a_{m+1} = \one$ are all the nodes of $\alg A$, then
\[
\alg A = \sum_{i=0}^m \alg A[a_i,a_{i+1}],
\]
where $\alg A[a_i,a_{i+1}]$ denotes the algebra with universe $\{\, b \in \alg A \mid a_i \leq b \leq a_{i+1} \,\}$. Therefore, by Corollary~\ref{cr-ngen2nnodes}, we obtain the following.

\begin{proposition}\label{pr-fgdecomp}
	Let $\alg A$ be a finitely generated nontrivial Heyting algebra. Then, for some $m \ge 0$,
	\begin{align}
	\alg A = \sum_{i=0}^{m} \alg B_i, \label{eq-nodec}
	\end{align}
	where each $\alg B_i$ is a nodeless nontrivial algebra.
\end{proposition}

\begin{definition}
We call a decomposition of the form~\eqref{eq-nodec} a \Def{nodeless decomposition}, and each algebra $\alg B_i$ is called a \Def{component} of this decomposition, with $\alg B_0$ and $\alg B_m$ being the \Def{starting} and \Def{ending} components, respectively. A decomposition is said to be \Def{finite} if it contains only finitely many components.
\end{definition}

Throughout the paper, $\two$ denotes the two-element Boolean algebra.

\begin{proposition} \label{pr-sinodeless}
	Let $\alg A$ be a finitely generated s.i.\ algebra. Then
	\begin{equation}\label{eq-pr-sinodeless}
		\alg A = \alg B_0 + \dots + \alg B_m + \two
	\end{equation}
	for some nodeless nontrivial algebras $\alg B_i$.
\end{proposition}

Observe that if $\alg A$ is as in \eqref{eq-pr-sinodeless}, then $\alg B_0 + \two$ is a subalgebra of $\alg A$, and for every $i > 0$, $\two + \alg B_i + \two$ is a subalgebra of $\alg A$. 
Note also that algebra $\alg A$ is finitely generated if and only if the algebras $\alg B_0 + \two$ and $\two + \alg B_i + \two$ (for all $i \in [1,m]$) are finitely generated.

\paragraph{Subdirect products.}

Recall from \cite[Chapter~3]{GraetzerB} that a nontrivial algebra is \Def{subdirectly irreducible} if the meet of all its nontrivial (that is, distinct from the identity) congruences is itself a nontrivial congruence.

It is a well-known fact that a Heyting algebra $\alg A$ is subdirectly irreducible (s.i.) if and only if $\alg A \cong \alg B + \two$ for some (possibly trivial) algebra $\alg B$; otherwise, it is a subdirect product of s.i.\ Heyting algebras. More precisely, the following holds.

\begin{proposition}
	Let $\alg A$ be a nontrivial Heyting algebra. Then $\alg A$ is a subdirect product of the algebras $\alg A/\theta_a$, where $a \in \alg A$ satisfies $a < \one$, and $\theta_a$ is a maximal congruence such that $(a,1) \notin \theta_a$.
\end{proposition}

Later, we will use the following particular corollary of the above proposition.

\begin{corollary}\label{cor-subdec}
	Let $\alg A$ be a Heyting algebra containing elements $a,b < \one$ such that for every $c \in \alg A$ distinct from $\one$, either $c \leq a$ or $c \leq b$. Then $\alg A$ is a subdirect product of the algebras $\alg A/\theta_a$ and $\alg A/\theta_b$.
\end{corollary}

\section{Finitely Generated Heyting Algebras: Definitions and Properties.}

In this section we recall several properties of finitely generated Heyting algebras that will be used later.

\subsection{One-generated Heyting Algebras.}

Analogously to the case of groups, for each $n \in \omega$, there exists a unique (up to isomorphism) one-generated (cyclic) Heyting algebra, denoted by $\ZZ_n$. In particular, $\ZZ_1$ is the trivial algebra, $\two \cong \ZZ_2$, $\PAlg_1 \cong \ZZ_7 \cong \ZZ_6 + \two$, and $\PAlg_2 \cong \two + \ZZ_7 \cong \two + \ZZ_6 + \two$. The algebra $\ZZ_\infty$ is freely generated by one element, and every finite cyclic Heyting algebra is a homomorphic image of $\ZZ_\infty$.

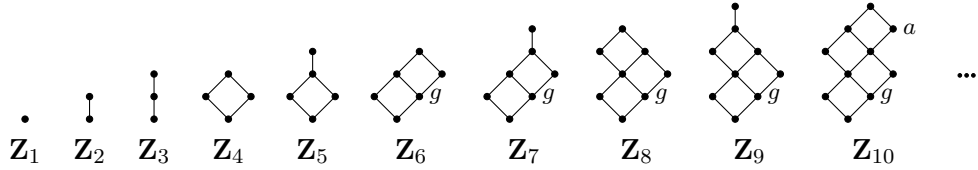
\begin{figure}[ht]
	\begin{tabular}{ccccccccccc}
		\begin{tikzpicture}[scale=0.6] 
			\draw[fill] (0,0) circle [radius=0.07];
		\end{tikzpicture}
		&	
		\begin{tikzpicture}[scale=0.6] 
			\draw[fill] (0,0) -- (0,0.5);	
			\draw[fill] (0,0) circle [radius=0.07];
			\draw[fill] (0,0.5) circle [radius=0.07]; 
		\end{tikzpicture}
		&
		\begin{tikzpicture}[scale=0.6] 
			\draw[fill] (0,0) -- (0,1);	
			\draw[fill] (0,0) circle [radius=0.07];
			\draw[fill] (0,0.5) circle [radius=0.07]; 
			\draw[fill] (0,1) circle [radius=0.07]; 
		\end{tikzpicture}
		&
		\begin{tikzpicture}[scale=0.6] 
			\draw[fill] (0,0) -- (-0.5,0.5);	
			\draw[fill] (0,0) -- (0.5,0.5);	
			\draw[fill] (0,1) -- (0.5,0.5);	
			\draw[fill] (0,1) -- (-0.5,0.5);	
			\draw[fill] (0,0) circle [radius=0.07];
			\draw[fill] (-0.5,0.5) circle [radius=0.07];
			\draw[fill] (0.5,0.5) circle [radius=0.07];
			\draw[fill] (0,1) circle [radius=0.07];				
		\end{tikzpicture}
		&
		\begin{tikzpicture}[scale=0.6] 
			\draw[fill] (0,0) -- (-0.5,0.5);	
			\draw[fill] (0,0) -- (0.5,0.5);	
			\draw[fill] (0,1) -- (0.5,0.5);	
			\draw[fill] (0,1) -- (-0.5,0.5);	
			\draw[fill] (0,1) -- (0,1.5);	
			\draw[fill] (0,0) circle [radius=0.07];
			\draw[fill] (-0.5,0.5) circle [radius=0.07];
			\draw[fill] (0.5,0.5) circle [radius=0.07];
			\draw[fill] (0,1) circle [radius=0.07];				
			\draw[fill] (0,1.5) circle [radius=0.07];	
		\end{tikzpicture}
		&	
		\begin{tikzpicture}[scale= 0.6] 
			\draw[fill] (0,0) -- (-0.5,0.5);	
			\draw[fill] (0,0) -- (1,1);	
			\draw[fill] (0,1) -- (0.5,0.5);	
			\draw[fill] (0.5,1.5) -- (-0.5,0.5);	
			\draw[fill] (1,1) -- (0.5,1.5);	
			\draw[fill] (0,0) circle [radius=0.07];
			\draw[fill] (-0.5,0.5) circle [radius=0.07];
			\draw[fill] (0.5,0.5) circle [radius=0.07];
			\draw[fill] (0,1) circle [radius=0.07];				
			\draw[fill] (0.5,1.5) circle [radius=0.07];
			\draw[fill] (1,1) circle [radius=0.07];
			\node[right]  at (0.5,0.5) {\footnotesize $g$};
		\end{tikzpicture}
		&
		\begin{tikzpicture}[scale=0.6] 
			\draw[fill] (0,0) -- (-0.5,0.5);	
			\draw[fill] (0,0) -- (1,1);	
			\draw[fill] (0,1) -- (0.5,0.5);	
			\draw[fill] (0.5,1.5) -- (-0.5,0.5);	
			\draw[fill] (1,1) -- (0.5,1.5);	
			\draw[fill] (0.5,1.5) -- (0.5,2);	
			\draw[fill] (0,0) circle [radius=0.07];
			\draw[fill] (-0.5,0.5) circle [radius=0.07];
			\draw[fill] (0.5,0.5) circle [radius=0.07];
			\draw[fill] (0,1) circle [radius=0.07];				
			\draw[fill] (0.5,1.5) circle [radius=0.07];
			\draw[fill] (1,1) circle [radius=0.07];
			\draw[fill] (0.5,2) circle [radius=0.07];
			\node[right]  at (0.5,0.5) {\footnotesize $g$};
		\end{tikzpicture}
		&
		\begin{tikzpicture}[scale=0.6] 
			\draw[fill] (0,0) -- (-0.5,0.5);	
			\draw[fill] (0,0) -- (1,1);	
			\draw[fill] (-0.5,1.5) -- (0.5,0.5);	
			\draw[fill] (0.5,1.5) -- (-0.5,0.5);	
			\draw[fill] (1,1) -- (0,2);	
			\draw[fill] (-0.5,1.5) -- (0,2);	
			\draw[fill] (0,0) circle [radius=0.07];
			\draw[fill] (-0.5,0.5) circle [radius=0.07];
			\draw[fill] (0.5,0.5) circle [radius=0.07];
			\draw[fill] (0,1) circle [radius=0.07];				
			\draw[fill] (0.5,1.5) circle [radius=0.07];
			\draw[fill] (1,1) circle [radius=0.07];
			\draw[fill] (-0.5,1.5) circle [radius=0.07];
			\draw[fill] (0,2) circle [radius=0.07];
			\node[right]  at (0.5,0.5) {\footnotesize $g$};
		\end{tikzpicture}
		& 
		\begin{tikzpicture}[scale=0.6] 
			\draw[fill] (0,0) -- (-0.5,0.5);	
			\draw[fill] (0,0) -- (1,1);	
			\draw[fill] (-0.5,1.5) -- (0.5,0.5);	
			\draw[fill] (0.5,1.5) -- (-0.5,0.5);	
			\draw[fill] (1,1) -- (0,2);	
			\draw[fill] (-0.5,1.5) -- (0,2);	
			\draw[fill] (0,2) -- (0,2.5);	
			\draw[fill] (0,0) circle [radius=0.07];
			\draw[fill] (-0.5,0.5) circle [radius=0.07];
			\draw[fill] (0.5,0.5) circle [radius=0.07];
			\draw[fill] (0,1) circle [radius=0.07];				
			\draw[fill] (0.5,1.5) circle [radius=0.07];
			\draw[fill] (1,1) circle [radius=0.07];
			\draw[fill] (-0.5,1.5) circle [radius=0.07];
			\draw[fill] (0,2) circle [radius=0.07];
			\draw[fill] (0,2.5) circle [radius=0.07];
			\node[right]  at (0.5,0.5) {\footnotesize $g$};
		\end{tikzpicture}
		& 
		\begin{tikzpicture}[scale=0.6] 
			\draw[fill] (0,0) -- (-0.5,0.5);	
			\draw[fill] (0,0) -- (1,1);	
			\draw[fill] (-0.5,1.5) -- (0.5,0.5);	
			\draw[fill] (1,2) -- (-0.5,0.5);	
			\draw[fill] (1,1) -- (0,2);	
			\draw[fill] (-0.5,1.5) -- (0,2);	
			\draw[fill] (0,2) -- (0.5,2.5);	
			\draw[fill] (1,2) -- (0.5,2.5);	
			\draw[fill] (0,0) circle [radius=0.07];
			\draw[fill] (-0.5,0.5) circle [radius=0.07];
			\draw[fill] (0.5,0.5) circle [radius=0.07];
			\draw[fill] (0,1) circle [radius=0.07];				
			\draw[fill] (0.5,1.5) circle [radius=0.07];
			\draw[fill] (1,1) circle [radius=0.07];
			\draw[fill] (-0.5,1.5) circle [radius=0.07];
			\draw[fill] (0,2) circle [radius=0.07];
			\draw[fill] (0.5,2.5) circle [radius=0.07];
			\draw[fill] (1,2) circle [radius=0.07];
			\node[right]  at (0.5,0.5) {\footnotesize $g$};
			\node[right]  at (1,2) {\footnotesize $a$};
		\end{tikzpicture}
		& 		\begin{tikzpicture}[scale=0.6] 
			\draw[fill,white] (0,0) circle [radius=0.04];
			\draw[fill] (0.3,1) circle [radius=0.04];
			\draw[fill] (0.15,1) circle [radius=0.04];
			\draw[fill] (0,1) circle [radius=0.04];
		\end{tikzpicture}
		\\
		$\ZZ_1$&$\ZZ_2$&$\ZZ_3$&$\ZZ_4$&$\ZZ_5$&$\ZZ_6$&$\ZZ_7$&$\ZZ_8$&$\ZZ_9$&$\ZZ_{10}$
		
	\end{tabular}
	\caption{One-generated Heyting algebras.} \label{fig-z}
\end{figure}

Notably, the algebras $\ZZ_1$–$\ZZ_5$ do not contain ordinary elements, whereas for all $n > 5$, the algebra $\ZZ_n$ contains a single ordinary element that generates the entire algebra (see Fig.~\ref{fig-z}, element~$g$). Thus, the following holds.

\begin{proposition}\label{pr-ordgen}
	If a Heyting algebra $\alg A$ contains an ordinary element, then it contains a subalgebra isomorphic to $\ZZ_n$, where $n \ge 6$.  
	
	Moreover, if $\alg A$ is s.i., then $\alg A$ contains subalgebras isomorphic to $\ZZ_{2k+1}$ and to $\ZZ_2 + \ZZ_{2k-3}$, where $k \ge 3$.
\end{proposition}

The algebra $\ZZ_7$ plays a special role in what follows, and we will need the following properties of this algebra.

\begin{proposition}\label{pr-Z7h}
	The algebra $\ZZ_7$ is a homomorphic image of every algebra $\ZZ_n$ with $n \ge 10$.
\end{proposition}
\begin{proof}
	In the algebra $\ZZ_{10}$, and in all $\ZZ_n$ with $n \ge 10$, the ideal $\pidl{a}$ is, as a lattice, isomorphic to $\ZZ_7$. Hence $\ZZ_7 \in \HH(\ZZ_{10})$.\\
\end{proof}

\begin{proposition}\label{pr-Z7}
	Let $\alg A$ be a Heyting algebra and $a \in \alg A$. If
	\begin{align}
		(\neg\neg a \to a) \to (\neg\neg a \lor \neg a) = \one \quad \text{and} \quad \neg a \lor \neg\neg a \neq \one, \label{eq-pr-Z7}
	\end{align}
	then the element $a$ generates a subalgebra isomorphic to $\ZZ_7$.
\end{proposition}

\begin{proof}
	Observe that $a$ is an ordinary element. Indeed, it cannot be dense, since in that case we would have $\one = \neg\neg a$, and hence, contrary to the assumption, $\neg a \lor \neg\neg a = \one$. Moreover, $a$ is not regular, for otherwise $\neg\neg a \to a = \one$, and from $(\neg\neg a \to a) \to (\neg\neg a \lor a) = \one$ it would again follow, contrary to the assumption, that $\neg a \lor \neg\neg a = \one$. Hence $a$ generates an algebra $\ZZ_n$ with $n \ge 6$.
	
	Next, note that $n \neq 6$, since in $\ZZ_6$ we have $\neg g \lor \neg\neg g = \one$, and $n < 8$, because for all $n \ge 8$,
	\[
	(\neg\neg g \to g) \to (\neg g \lor \neg\neg g) \neq \one.
	\]
	Therefore, $n = 7$.
\end{proof}

\bigskip

Recall from \cite{Kuznetsov_1973} that every nontrivial finitely generated Heyting algebra $\alg A$ contains the smallest dense element, namely, 
\[
d = (g_1 \lor \neg g_1) \land \dots \land (g_n \lor \neg g_n),
\]
where $g_i$, $i \in [1,n]$, is a set of generators of $\alg A$ (a proof can be found, e.g., in \cite{Bezh_Grig_LocFin} or \cite{Citkin_FinGen_2023}).

\begin{proposition}\label{pr-nodeZ7}
	Let $\VV$ be a variety of Heyting algebras, and let $\alg B$ be a finitely generated nodeless algebra such that $\alg B + \ZZ_2 \in \VV$. Suppose that $d \in \alg B$ is the smallest dense element, and that the lattice $\alg B\pfltr{d}$ contains a nontrivial node $b$ (see Fig.~\ref{fig-dense}). Then the algebra $\alg B + \ZZ_2$ contains a subalgebra isomorphic to $\ZZ_2 + \ZZ_7$. 
\end{proposition}

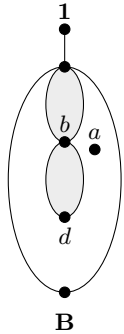
\begin{figure}[ht]
	\begin{center}
		\begin{tikzpicture}
			\draw (0,0) ellipse (0.75cm and 1.5cm);
			\filldraw[fill=black!7] (0,1) ellipse (0.25cm and 0.5cm);
			\filldraw[fill=black!7] (0,0) ellipse (0.25cm and 0.5cm);
			\draw[fill] (0,1.5) -- (0,2);
			\draw[fill] (0,2) circle [radius=0.07];
			\draw[fill] (0,1.5) circle [radius=0.07];
			\draw[fill] (0,0.5) circle [radius=0.07];
			\draw[fill] (0,-0.5) circle [radius=0.07];
			\draw[fill] (0,-1.5) circle [radius=0.07];
			\draw[fill] (0.4,0.4) circle [radius=0.07];
			\node[above]  at (0.4,0.4) {\footnotesize $a$};
			\node[above]  at (0,2) {\footnotesize $\one$};
			\node[above]  at (0,0.5) {\footnotesize $b$};
			\node[below]  at (0,-0.5) {\footnotesize $d$};
			\node[below]  at (0,-1.65) {\footnotesize $\alg B$};
		\end{tikzpicture}	
		\caption{Dense node.}\label{fig-dense}
	\end{center}
\end{figure}

\begin{proof}
	Since $\alg B$ is nodeless, element $b$ cannot be a node in $\alg B$ and hence, there exists an element $a \in \alg B + \ZZ_2$ that is incomparable with $b$ and clearly not dense. 
	
	Let 
	\[
	c = (\neg\neg a \to d) \land \neg\neg a.
	\]
	We show that $c$ satisfies the conditions of~\eqref{eq-pr-Z7}, namely
	\[
	(\neg\neg c \to c) \to (\neg c \lor \neg\neg c) = \one 
	\quad\text{and}\quad 
	\neg c \lor \neg\neg c \neq \one.
	\]
	
	First, note that as $d$ is dense,
	\[
	\neg c = \neg((\neg\neg a \to d) \land \neg\neg a) = \neg a,
	\quad\text{and hence}\quad
	\neg\neg c = \neg\neg a.
	\]
	Thus,
	\[
	\neg\neg c \to c 
	= \neg\neg a \to ((\neg\neg a \to d) \land \neg\neg a)
	= \neg\neg a \to (\neg\neg a \to d)
	= \neg\neg a \to d.
	\]
	Therefore, it suffices to verify that
	\[
	(\neg\neg a \to d) \to (\neg a \lor \neg\neg a) = \one 
	\quad\text{and}\quad 
	\neg a \lor \neg\neg a \neq \one.
	\]
	
	Since $a \neq \zero$, $\neg a \neq \one$; and because $a$ is not dense, $\neg\neg a \neq \one$. As the algebra $\alg B + \ZZ_2$ is subdirectly irreducible, it follows that $\neg a \lor \neg\neg a \neq \one$.
	
	To prove $(\neg\neg a \to d) \to (\neg a \lor \neg\neg a) = \one$, we show that:
	\begin{itemize}
		\item[(i)] $(\neg\neg a \to d) \leq b$, and
		\item[(ii)] $b \leq (\neg a \lor \neg\neg a)$.
	\end{itemize}
	
	(i) Since $d$ is dense and $d \leq (\neg\neg a \to d)$, the element $\neg\neg a \to d$ is also dense. As $b$ is a node, $\neg\neg a \to d$ and $b$ are comparable. Moreover, $d$ and $b$ are both nodes in $\alg B\pfltr{d}$, and since $d < b$, Corollary~\ref{cor-nodesimpl} yields $b \to d = d$. Hence,
	\[
	b \to (\neg\neg a \to d) = \neg\neg a \to (b \to d) = \neg\neg a \to d.
	\]
	If $\neg\neg a \to d = \one$, we would have $a \leq \neg\neg a \leq d \leq b$, contradicting the incomparability of $a$ and $b$. Thus $b \not\leq (\neg\neg a \to d)$, and since $b$ is a node, $\neg\neg a \to d \leq b$.
	
	(ii) The element $(\neg a \lor \neg\neg a)$ is dense, and hence comparable with $b$. However, $(\neg a \lor \neg\neg a) \leq b$ would imply $a \leq \neg\neg a \leq \neg a \lor \neg\neg a \leq b$, again contradicting the incomparability of $a$ and $b$. Therefore $b \leq (\neg a \lor \neg\neg a)$.
\end{proof}

We conclude this section with an observation on nodeless algebras.

\begin{proposition}\label{pr-ndldense}
	Let $\alg A$ be a nodeless Heyting algebra. Then:
	\begin{itemize}
		\item[(a)] if $\alg A$ has exactly one dense element, then it is Boolean;
		\item[(b)] if $\alg A$ has exactly two dense elements, then it contains a subalgebra isomorphic to $\ZZ_6$; 
		\item[(c)] if $\alg A$ is finite and the dense elements of $\alg A$ form a four-element Boolean lattice, then $\alg A$ either contains a subalgebra isomorphic to $\ZZ_6$, or is a subdirect product of the algebras $\alg B + \two$ and $\alg C + \two$, where $\alg B$ and $\alg C$ are Boolean algebras. 
	\end{itemize}
\end{proposition}

\begin{proof}
	(a) Since $a \lor \neg a$ is always dense, $\alg A \models x \lor \neg x \approx \one$, and thus $\alg A$ is Boolean.
	
	(b) Let $\{d, \one\}$ be the set of dense elements of $\alg A$. Since $\alg A$ is nodeless, it contains an element incomparable with $d$, and by Corollary~\ref{cor-dnr}, $\alg A$ contains an ordinary element.  
	
	Let $a \in \alg A$ be an ordinary element. Then, by Proposition~\ref{pr-ordgen}, $a$ generates in $\alg A$ a cyclic subalgebra $\ZZ_n$ with $n \ge 6$. Since every algebra $\ZZ_k$ with $k > 6$ has more than two dense elements, it follows that $n = 6$.   
	
	(c) Suppose that $d, a_1, a_2, \one \in \alg A$ are the dense elements, with $d$ the smallest. By Proposition~\ref{pr-coatZ7}, if $\alg A$ contains regular coatoms, then it contains a subalgebra isomorphic to $\ZZ_6$. Thus, if $\ZZ_6 \notin \CSub(\alg A)$, all coatoms in $\alg A$ are dense, and by assumption, there are only two coatoms. Therefore, for any $c \in \alg A$ with $c < \one$, we have $c \leq a_1$ or $c \leq a_2$. Then, by Corollary~\ref{cor-subdec}, $\alg A$ is a subdirect product of its two s.i.\ homomorphic images. 
	
	These homomorphic images belong to the variety $\VV$ generated by $\alg A$. Note that $\alg A$ contains a three-element chain subalgebra but does not contain a four-element chain subalgebra. Hence, the four-element chain subalgebra is not in $\VV$, and by \cite{Hosoi_Ono_Intermediate_1970}, the s.i.\ algebras from $\VV$ are of the form $\alg B + \two$, where $\alg B$ is a Boolean algebra.
\end{proof}

\subsection{Finitely Generated Heyting Algebras.}

The following proposition generalizes the theorem from \cite{Kuznetsov_1973}, whose proof can be found in \cite{Citkin_FinGen_2023}.

\begin{proposition}[{\cite[Theorem 23]{Citkin_FinGen_2023}}] \label{pr-fgen}
	Let $\alg A$ be an $n$-generated Heyting algebra and $a \in \alg A$. Then the algebra $\alg A\pfltr{a}$ is generated by at most $2n+1$ elements.
\end{proposition}

Using Proposition \ref{pr-fgen} and Corollary \ref{cr-ngen2nnodes} we can straighten Propositions \ref{pr-fgdecomp}. 

\begin{corollary}\label{cor-decomp}

	Let $\alg A$ be an $n$-generated nontrivial Heyting algebra. Then $\alg A$ has a nodeless decomposition containing at most $2n + 2$ components, each of which is generated by at most 
	\[
	n + m(2n + 1) \le n + (2n + 2)(2n + 1) \le (2n + 2)^2
	\]
	elements.
\end{corollary}

We will also need the following properties of finitely generated Heyting algebras, which can be found, for example, in \cite{Citkin_FinGen_2023}.

\begin{proposition}\label{pr-fg}
	Let $\alg A$ be an $n$-generated Heyting algebra. Then:
	\begin{itemize}
		\item[(a)] $\alg A$ contains the smallest dense element $d$;
		
		\item[(b)] there exist at most $2^{2^n}$ distinct cosets $a/\pfltr{d}$, where $a \in \alg A$;
		
		\item[(c)] for each $a \in \alg A$, the cardinality of the coset $a/\pfltr{d}$ does not exceed that of the filter $\pfltr{d}$;
		
		\item[(d)] each coset $a/\pfltr{d}$ contains the smallest element $d \land a$ and the largest element $d \to a$;
		
		\item[(e)] each coset $a/\pfltr{d}$ contains a unique regular element, namely $d \to a$.
	\end{itemize}
\end{proposition}

\begin{corollary}\label{cor-ord}
	Let $\alg A$ be a finitely generated nodeless Heyting algebra. Then $\alg A$ is Boolean if and only if it does not contain ordinary elements. 
\end{corollary}

\begin{proof}
	In Boolean algebras, all elements are regular and therefore not ordinary.
	
	Conversely, we show that if a finitely generated nodeless Heyting algebra $\alg A$ is not Boolean, then $\alg A$ contains an ordinary element.
	
	If all $\pfltr{d}$-cosets distinct from $\pfltr{d}$ are singletons, then by Lemma \ref{pr-fg}(a), $d$ would be a nontrivial node, contradicting the assumption. Hence, there exists a coset distinct from $\pfltr{d}$ that is not a singleton.
	
	Let $a \in \alg A$ with $a \notin \pfltr{d}$, where $d$ is the smallest dense element of $\alg A$. If the coset $a/\pfltr{d}$ contains more than one element, then by Proposition~\ref{pr-fg}(e), the coset $a/\pfltr{d}$ contains a unique regular element $d \to a$, which, by Proposition~\ref{pr-fg}(d), is distinct from the element $d \land a$. The latter is neither regular nor dense, and therefore it is ordinary.
\end{proof}

\begin{corollary}\label{cor-fin}
	Let $\alg A$ be a finitely generated Heyting algebra, and let $d \in \alg A$ be its smallest dense element. Then $\alg A$ is finite if and only if the filter of all its dense elements, $\pfltr{d}$, is finite.
\end{corollary}

Recall that a variety $\VV$ is called \Def{locally finite} if every finitely generated algebra in $\VV$ is finite.

\begin{theorem}\label{th-lf}
	Variety of Heyting algebra $\VV$ is locally finite if and only if its subvariety $\VV' \subseteq \VV$ defined by the equation $\neg x \lor \neg\neg x \approx \one$ is locally finite.
\end{theorem}
\begin{proof}
	It should be clear that if variety $\VV$ is locally finite, each of its subvarieties is also locally finite, thus if $\VV$ is locally finite, so is $\VV'$.
	
	Conversely, suppose that subvariety $\VV'$ defined by equation $\neg x \lor \neg\neg x \approx \one$ is locally finite, and let $\alg A \in \VV$ be a finitely generated algebra with $d \in \alg A$ being its smallest dense element. Then, by Proposition \ref{pr-fgen}, algebra $\alg A\pfltr{d}$ is finitely generated. Let us observe that algebra $\alg A' := \two + \alg A\pfltr{d}$ is (isomorphic to) a subalgebra of $\alg A$; hence, $\alg A' \in \VV$. It is easy to see that all elements of $\alg A'$, which are distinct from $\zero$,   are dense and therefore, $\alg A' \models \neg x \lor \neg\neg x \approx \one$; therefore, $\alg A' \in \VV'$. Moreover, as algebra $\alg A\pfltr{d}$ is finitely generated, algebra $\two +\alg A\pfltr{d}$ is also finitely generated: it is generated by generators of $\alg A\pfltr{d}$ and by the smallest nontrivial node of $\alg A'$; hence, because variety $\VV'$ is locally finite  by assumption, algebra $\two + \alg A\pfltr{d}$ is finite and consequently, filter $\pfltr{d}$ is finite. Therefore, by Corollary \ref{cor-fin}, $\alg A$ is finite.  
\end{proof}

\paragraph{Criterion of local finiteness}

Recall from~\cite{Bezh_G_LocFin} that a class~$\KK$ is said to be \Def{uniformly locally finite in the weak sense} if there exists a function~$\beta$ on~$\mathbb{N}$ such that, for any~$n \in \mathbb{N}$, the cardinalities of all $n$-generated algebras from~$\KK$ are bounded by~$\beta(n)$.

Later we will employ the following criterion of local finiteness:

\begin{theorem}{\cite[Theorem~3.7.4]{Bezh_G_LocFin}}
	\label{th-crlf}
	Let $\VV$ be a variety of finite type. Then $\VV$ is locally finite if and only if the class of its s.i.\ algebras is uniformly locally finite in the weak sense.  
\end{theorem}

\section{Primitive Varieties: Definitions and Properties.}	\label{sec-pdpr}

In this section, we recall the notion of primitive varieties and their properties; further details can be found in \cite{GorbunovBookE}.

If $\KK$ is a class of algebras, denote by $\Var(\KK)$ the smallest variety containing $\KK$ and by $\QVar(\KK)$ -- the smallest quasivariety containing $\KK$, that is, $\Var(\KK) = \HH\CSub\PP(\KK)$ and $\QVar(\KK) = \CSub\PP\PP_u(\KK)$. \\

The following proposition is a corollary of a much more general result given in \cite[Corollary~2.1.16]{GorbunovBookE}.

\begin{proposition}\label{pr-incl}
	If $\alg A$ is a finite s.i. algebra of finite type and $\alg B$ is an algebra of the same type, then $\alg A \in \QVar(\alg B)$ if and only if $\alg A \in \CSub(\alg B)$
\end{proposition}

\begin{definition}
	Let $\VV$ be a variety. Then $\VV$ is said to be \Def{structurally complete} if, for any quasivarieties $\QQ_1$ and $\QQ_2$ with $\VV = \Var(\QQ_1) = \Var(\QQ_2)$, it follows that $\QQ_1 = \QQ_2$.  
	Moreover, $\VV$ is called \Def{primitive} (or \Def{hereditarily structurally complete}) if $\VV$ and all its subvarieties are structurally complete.
\end{definition}

\paragraph{Weakly projective algebras.}

The notion of weakly projective algebras plays a central role in the study of primitive locally-finite varieties.

\begin{definition}[see, e.g.,{\cite[Section 5.1.4]{GorbunovBookE}}]\label{def-wp}
	Let $\KK$ be a class of algebras and $\alg A \in \KK$. Then the algebra $\alg A$ is called \Def{weakly projective} in $\KK$ (or \Def{weakly $\KK$-projective}) if, for any $\alg B \in \KK$, the condition $\alg A \in \HH(\alg B)$ entails $\alg A \in \II\CSub(\alg B)$.
\end{definition}

Weak projectivity is relative to a class: an algebra may fail to be weakly projective in a class while being weakly projective in one of its subclasses. This observation justifies the following definition. 

\begin{definition}\label{def-totnp}
	An algebra $\alg A$ is called \Def{totally non-projective} if $\alg A$ is not weakly projective in the variety $\Var(\alg A)$ that it generates.
\end{definition}
Clearly, a totally non-projective algebra is not weakly projective in any variety containing it. As we shall see, all prohibited algebras are totally non-projective.

\begin{proposition}\label{pr-totnp}
	Let $\alg A$ be a totally non-projective finite s.i.\ algebra of finite type. Then $\Var(\alg A)$ is not structurally complete.
\end{proposition}

\begin{proof}
	Since $\alg A$ is totally non-projective, there exists an algebra $\alg B \in \Var(\alg A)$ such that $\alg A \in \HH(\alg B)$ but $\alg A \notin \CSub(\alg B)$. 
	As $\alg B \in \Var(\alg A)$, we have $\Var(\alg B) \subseteq \Var(\alg A)$. On the other hand, because $\alg A \in \HH(\alg B)$, it follows that $\alg A \in \Var(\alg B)$, and consequently $\Var(\alg A) \subseteq \Var(\alg B)$. Hence, $\Var(\alg A) = \Var(\alg B)$.
	
	Since $\alg A$ is a finite s.i.\ algebra of finite type and $\alg A \notin \CSub(\alg B)$, by Proposition~\ref{pr-incl} we have $\alg A \notin \QVar(\alg B)$. Therefore, $\QVar(\alg A) \neq \QVar(\alg B)$, while $\Var(\alg A) = \Var(\alg B)$, which shows that $\Var(\alg A)$ is not structurally complete.
\end{proof}

As a corollary, we obtain a necessary condition for primitivity in varieties of finite type.

\begin{corollary}\label{cor-nonpr}
	A variety $\VV$ of finite type containing a totally non-projective finite s.i.\ algebra is not primitive. 
\end{corollary}

\begin{remark}
	Proposition~\ref{pr-totnp} and Corollary~\ref{cor-nonpr} also hold for finitely presentable s.i.\ totally non-projective algebras.
\end{remark}

\subsection{Projective algebras.}

In this section, we recall the notion of projective algebras and establish some of their properties used in the proof of the main theorem.

\begin{definition}[{see, e.g., \cite[\S83]{GraetzerB}}]
	Let $\VV$ be a variety. An algebra $\alg A \in \VV$ is called \Def{projective} in $\VV$ (or \Def{$\VV$-projective}, for short) if for any algebras $\alg A_1, \alg A_2 \in \VV$, any surjective homomorphism $h: \alg A_1 \to \alg A_2$, and any homomorphism $f: \alg A \to \alg A_2$, there exists a homomorphism $g: \alg A \to \alg A_1$ such that $h g = f$.
\end{definition}

For instance, any free algebra in $\VV$ is $\VV$-projective. It is also clear that every $\VV$-projective algebra is weakly $\VV$-projective.

\begin{definition}[{see \cite[\S82]{GraetzerB}}]
	An algebra $\alg B$ is called a \Def{retract} of an algebra $\alg A$ if there exist a homomorphism $h: \alg B \to \alg A$ and an epimorphism $g: \alg A \to \alg B$ such that the composition $g h: \alg B \to \alg B$ is the identity $id_\alg B$. The homomorphism $h$ is called a \Def{retraction}. In fact, the $\VV$-projective algebras are exactly the retracts of the free algebras in~$\VV$.
\end{definition}

For given algebras $\alg A, \alg B \in \VV$ and a surjective homomorphism $h: \alg A \to \alg B$, to show that $\alg B$ is a retract, it suffices to choose, in each coset $a/h$ with $a \in \alg A$, a single element $a_h$ such that the set $\set{a_h}{a \in \alg A}$ forms a subalgebra. Indeed, if we define a map $g: \alg B \to \alg A$ by $g(h(a)) = a_h$, it is not hard to see that $g$ is a homomorphism and that $h g$ is the identity map.

\begin{figure}[ht]
	\centering
	\begin{tabular}{ccc}
		\begin{tikzpicture}[scale=1] 
			\draw[fill] (0,0) -- (0,1);	
			\draw[fill] (0,0) circle [radius=0.07];
			\draw[fill] (0,0.5) circle [radius=0.07]; 
			\draw[fill] (0,1) circle [radius=0.07]; 
			\node[below]  at (0,0) {\footnotesize $\zero$};
			\node[above]  at (0,1) {\footnotesize $\one$};
			\node[right]  at (0,0.5) {\footnotesize $a$};
		\end{tikzpicture}
		& $\quad$ &
		\begin{tikzpicture}[scale=1] 
			\draw[fill] (0,0) -- (-0.5,0.5);	
			\draw[fill] (0,0) -- (0.5,0.5);	
			\draw[fill] (0,1) -- (0.5,0.5);	
			\draw[fill] (0,1) -- (-0.5,0.5);	
			\draw[fill] (0,1) -- (0,1.5);	
			\draw[fill] (0,0) circle [radius=0.07];
			\draw[fill] (-0.5,0.5) circle [radius=0.07];
			\draw[fill] (0.5,0.5) circle [radius=0.07];
			\draw[fill] (0,1) circle [radius=0.07];				
			\draw[fill] (0,1.5) circle [radius=0.07];	
			\draw[fill] (0,1) circle [radius=0.07]; 
			\node[below]  at (0,0) {\footnotesize $\zero$};
			\node[above]  at (0,1.5) {\footnotesize $\one$};
			\node[right]  at (0.5,0.5) {\footnotesize $a$};
		\end{tikzpicture}
	\end{tabular}
	\caption{}\label{fig_Z35}
\end{figure}

\begin{example}\label{ex-proj_2_4}
	The Heyting algebras $\ZZ_3$ and $\ZZ_5$ (see Fig.~\ref{fig_Z35}) are projective in $\HA$. Indeed, suppose that $h: \alg A \to \ZZ_3$ is a homomorphism. Then the map $g$ defined by $g(\one) = \one$, $g(\zero) = \zero$, and $g(a) = a' \lor \neg a'$, where $a'$ is any element from $h^{-1}(a)$, is a retraction. 
	
	Similarly, if $h: \alg A \to \ZZ_5$ is a homomorphism of $\alg A$ onto $\ZZ_5$, then the map $g$ defined by 
	\[
	g(\zero) = \zero, \quad g(a) = \neg\neg a', \quad g(\neg a) = \neg a', \quad g(a \lor \neg a) = \neg\neg a' \lor \neg a', \quad g(\one) = \one,
	\]
	where $a'$ is any element from $h^{-1}(a)$, is a retraction.\qed
\end{example}

Finite projective Heyting algebras (that is, algebras projective in the entire variety of Heyting algebras~$\HA$) were described in~\cite{Balbes_Horn_1970}. They are precisely the finite s.i.\ Heyting algebras that are sums of the algebras $\ZZ_2$ and $\ZZ_4$. More information about projective Heyting algebras can be found in~\cite{Ghilardi_Zawadowski_Book_2002}. In particular, it was observed there that finitely generated projective Heyting algebras are exactly the finitely generated subalgebras of free algebras (see~\cite[p.~109]{Ghilardi_Zawadowski_Book_2002}). Very little is known about projective algebras in proper subvarieties of~$\HA$.

\begin{figure}[ht]
	\centering
	\begin{tabular}{ccc}
		
		\begin{tikzpicture}
			\draw (0,0) ellipse (1cm and 1.75cm);
			\draw(0,1.42) circle (10pt);
			\draw(0,0.5) circle (10pt);
			\draw (0,-1.0) ellipse (0.5cm and 0.75cm);
			\draw(0,-0.6) circle (10pt);
			\draw[fill] (0,0.81) -- (0,1.1);
			\draw[fill] (0,-0.25) -- (0,0.15);
			\draw[fill] (0,1.8) circle [radius=0.07];
			\draw[fill] (0,0.5) circle [radius=0.07];
			\draw[fill] (0,-0.5) circle [radius=0.07];
			\draw[fill] (0,-1.71) circle [radius=0.07];
			\node[above]  at (0,1.8) {\footnotesize $\one$};
			\node[left]  at (0.1,0.5) {\footnotesize $b'$};
			\node[below]  at (0,-0.5) {\footnotesize $a'$};
			
			\draw (4,-1.0) ellipse (0.5cm and 0.75cm);
			\draw[fill] (4,-0.25) -- (4,1.8);
			\draw[fill] (4,1.8) circle [radius=0.07];
			\draw[fill] (4,0.5) circle [radius=0.07];
			\draw[fill] (4,-1.71) circle [radius=0.07];
			\draw[fill] (4,-0.25) circle [radius=0.07];
			\node[above]  at (4,1.8) {\footnotesize $\one$};
			\node[right]  at (4,-0.2) {\footnotesize $a$};
			\node[right]  at (4,0.5) {\footnotesize $b$};
			\node[below]  at (4,-1.8) {\footnotesize $\alg B$};
			\node[below]  at (0,-1.8) {\footnotesize $\alg F$};
			\node[above]  at (2,1.7) {\footnotesize $h$};
			\node[above]  at (2,0.4) {\footnotesize $h$};
			\node[above]  at (2,-0.35) {\footnotesize $h$};
			\node[above]  at (2,-1.85) {\footnotesize $h$};
			\node[below]  at (2,-0.28) {\footnotesize $g'$};
			\node[below]  at (4,-0.8){\footnotesize $\alg A$};
			\draw[-{Stealth[length=5pt]}] (0,1.8) -- (4,1.8);
			\draw[-{Stealth[length=5pt]}] (0,0.5) -- (4,0.5);
			\draw[-{Stealth[length=5pt]}] (0,-0.25) -- (4,-0.29);
			\draw[-{Stealth[length=5pt]}] (0,-1.8) -- (4,-1.8);
			\draw[-{Stealth[length=5pt]}] (4,-0.3) -- (0,-0.5);
		\end{tikzpicture}
		& $\quad$ &
		\begin{tikzpicture}
			\draw (0,0) ellipse (1cm and 1.75cm);
			\draw(0.6,0.1) circle (10pt);
			\draw(-0.6,0.1) circle (10pt);
			\draw(0,0.65) circle (10pt);
			\draw (0,-1.0) ellipse (0.5cm and 0.75cm);
			\draw(0,-0.6) circle (10pt);
			\draw(0,1.42) circle (10pt);
			\draw[fill] (0,-0.25) -- (-0.25,0);
			\draw[fill] (0,-0.25) -- (0.25,0);
			\draw[fill] (0,0.3) -- (-0.25,0.1);
			\draw[fill] (0,0.3) -- (0.25,0.1);
			\draw[fill] (0,1) -- (0,1.1);
			\draw[fill] (0,1.8) circle [radius=0.07];
			\draw[fill] (0.6,0.15) circle [radius=0.07];
			\draw[fill] (0,-0.5) circle [radius=0.07];
			\draw[fill] (0,-1.71) circle [radius=0.07];
			\node[above]  at (0,1.8) {\footnotesize $\one$};
			\node[left]  at (0.68,0.18) {\footnotesize $b'$};
			\node[below]  at (0,-0.5) {\footnotesize $a'$};
			
			\draw (4,-1.0) ellipse (0.5cm and 0.75cm);
			\draw[fill] (4,-0.25) -- (4.5,0.25);
			\draw[fill] (4,0.85) -- (4.5,0.3);
			\draw[fill] (4,-0.25) -- (3.5,0.25);
			\draw[fill] (4,0.85) -- (3.5,0.3);
			\draw[fill] (4,0.8) -- (4,1.8);
			\draw[fill] (4,1.8) circle [radius=0.07];
			\draw[fill] (4.5,0.25) circle [radius=0.07];
			\draw[fill] (4,-1.71) circle [radius=0.07];
			\draw[fill] (4,-0.25) circle [radius=0.07];
			\draw[fill] (3.5,0.25) circle [radius=0.07];
			\draw[fill] (4,0.85) circle [radius=0.07];
			\node[above]  at (4,1.8) {\footnotesize $\one$};
			\node[right]  at (4,-0.2) {\footnotesize $a$};
			\node[left]  at (3.7,0.) {\footnotesize $b$};
			\node[below]  at (4,-1.8) {\footnotesize $\alg C$};
			\node[below]  at (0,-1.8) {\footnotesize $\alg F$};
			\node[above]  at (2,1.7) {\footnotesize $h$};
			\node[above]  at (2,0.2) {\footnotesize $h$};
			\node[above]  at (2,-0.35) {\footnotesize $h$};
			\node[above]  at (2,-1.85) {\footnotesize $h$};
			\node[below]  at (2,-0.26) {\footnotesize $g'$};
			\node[below]  at (4,-0.8){\footnotesize $\alg A$};
			
			
			\draw[-{Stealth[length=5pt]}] (0,1.8) -- (4,1.8);
			\draw[-{Stealth[length=5pt]}] (0.55,0.15) -- (3.5,0.25);
			\draw[-{Stealth[length=5pt]}] (0,-0.27) -- (4,-0.29);
			\draw[-{Stealth[length=5pt]}] (0,-1.8) -- (4,-1.8);
			\draw[-{Stealth[length=5pt]}] (4,-0.3) -- (0,-0.5);
		\end{tikzpicture}
		\\
		(a) && (b)	
	\end{tabular}
	\caption{Projectivity}\label{fig_proj}
\end{figure}
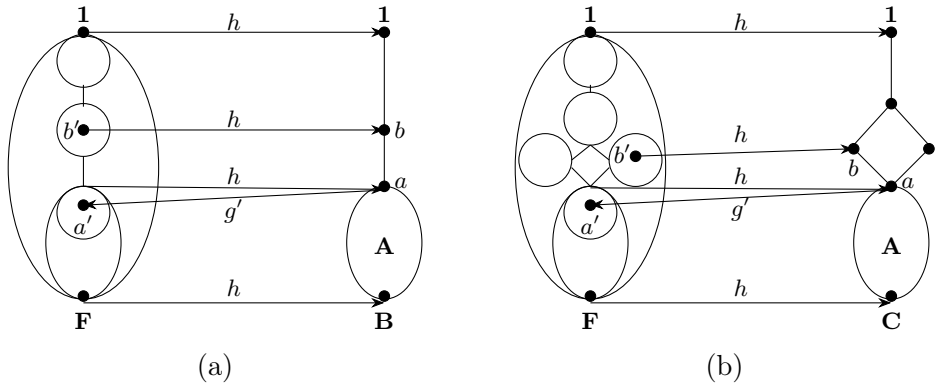

\begin{theorem}\label{th-proj}
	Let $\VV$ be a variety of Heyting algebras, and suppose that algebra $\alg A + \ZZ_2 \in \VV$ is $\VV$-projective. Then the following holds:
	\begin{itemize}
		\item[(a)] if $\alg B = \alg A + \ZZ_2 + \ZZ_2 \in \VV$, then $\alg B$ is $\VV$-projective;
		\item[(b)] if $\alg C = \alg A + \ZZ_4 + \ZZ_2 \in \VV$, then $\alg C$ is $\VV$-projective.
	\end{itemize}
\end{theorem}

\begin{proof}
	To prove the theorem, we show that the algebras $\alg B$ and $\alg C$ are retracts of a free algebra in $\VV$. To that end, for a given homomorphism from a free algebra $\alg F$ onto $\alg B$ or $\alg C$, we construct a corresponding retraction.
	
	(a) Let $h: \alg F \map \alg B$ be a homomorphism of a free algebra $\alg F$ onto $\alg B$. It is easy to see that $\alg A + \ZZ_2$ is a homomorphic image of $\alg A + \ZZ_2 + \ZZ_2$, so the homomorphism $h$ can be extended to a homomorphism $h': \alg F \map \alg A + \ZZ_2$. By assumption, the algebra $\alg A + \ZZ_2$ is $\VV$-projective, and hence there exists a retraction $g: \alg A + \ZZ_2 \map \alg F$.
	
	To show that $\alg A + \ZZ_2 + \ZZ_2$ is a retract of $\alg F$, in every coset $h^{-1}(d)$, $d \in \alg B$, we choose an element so that the chosen elements form a subalgebra of $\alg F$. Observe that, except for the coset $\one/h$, all $h'$-cosets coincide with the $h$-cosets. In $\one/h$, we choose $\one$, and since $\alg A + \two$ is projective, we can select the remaining elements so that they form a subalgebra of $\alg F$. It remains only to select an element $b'$ in the coset $h^{-1}(b)$ (see Fig.~\ref{fig_proj}(a)) so that all previously chosen elements together with $b'$ form a subalgebra. 
	
	Let $a' \in h^{-1}(a)$ be an element chosen in the coset that $h$ maps to the greatest element of $\alg A$, and let $b' \in h^{-1}(b)$ be any element from $h^{-1}(b)$. Now consider the element $b' = c \lor (c \to a')$.
	
	Since
	\[
	h(b') = h(c \lor (c \to a')) = h(c) \lor (h(c) \to h(a')) = b \lor (b \to a) = b \lor a = b,
	\]
	we have $b' \in h^{-1}(b)$. Moreover, by~\eqref{eq-node}, we have $b' \to a' = a'$, and using~\eqref{eq-strg}, one can verify that all the chosen elements form a subalgebra of $\alg F$. Hence, $\alg A + \ZZ_2 + \two$ is a projective algebra.
	
	(b) Now, let us show that $\alg A + \ZZ_4 + \two$ is projective. We use an argument similar to that in the previous case. Select an element $a'$ from the coset mapped to the greatest element of $\alg A$, and any element $b'$ from $h^{-1}(b)$ (see Fig.~\ref{fig_proj}(b)). Then, by Proposition~\ref{pr-fourpr}, the elements 
	\[
	a',\quad b' \to a',\quad (b' \to a') \to a',\quad (b' \to a') \lor ((b' \to a') \to a')
	\]
	form a four-element relatively pseudocomplemented sublattice that is lattice-isomorphic to $\ZZ_4$. Therefore, the selected elements form a subalgebra of $\alg F$.
\end{proof}

\begin{example}
	From Theorem~\ref{th-proj} and Example~\ref{ex-proj_2_4}, it follows that any Heyting algebra of the form
	\[
	\alg B_0 + \alg B_1 + \dots + \alg B_n + \ZZ_2,
	\]
	where $\alg B_i \in \{\ZZ_2, \ZZ_4\}$ for all $i \in [0,n]$, is projective in $\HA$. 
	This provides an alternative proof of Theorem~4.10 from \cite{Balbes_Horn_1970}.
\end{example}

Next, we present examples of projective algebras in certain proper subvarieties of~$\HA$.

\begin{corollary}\label{cor-prof}
	Let $\VV$ be a variety of Heyting algebras having an s.i.\ free algebra $\alg F = \alg A + \ZZ_2$. 
	Then all algebras of the form
	\[
	\alg A + \alg B_1 + \dots + \alg B_n + \ZZ_2,
	\]
	where $\alg B_i \in \{\ZZ_2, \ZZ_4\}$ for all $i \in [1,n]$, are $\VV$-projective.
\end{corollary}

\begin{corollary}\label{cor-Med}
Let $\Md$ be the variety corresponding to the Medvedev Logic. Then algebra $\ZZ_8 + \two$ is free in $\Md$ (see \cite{Wojtylak_Problem_2004}). 	Then all algebras of the form
\[
\alg A + \alg B_1 + \dots + \alg B_n + \ZZ_2,
\]
where $\alg A \in \{\ZZ_1, \ZZ_2, \ZZ_4, \ZZ_8\}$ and $\alg B_i \in \{\ZZ_2, \ZZ_4\}$ for all $i \in [1,n]$, are $\Md$-projective.   
\end{corollary}

\begin{corollary}\label{cor-Scott}
	Let $\Sc$ be the variety corresponding to Scott logic, that is, $\Sc$ is the largest variety of Heyting algebras omitting $\ZZ_7$.  
	Then all algebras of the form
	\[
	\alg A + \alg B_1 + \dots + \alg B_n + \ZZ_2,
	\]
	where $\alg A \in \{\ZZ_1, \ZZ_2, \ZZ_4, \ZZ_8\}$ and $\alg B_i \in \{\ZZ_2, \ZZ_4\}$ for all $i \in [1,n]$, are $\Sc$-projective.   
\end{corollary}

\begin{proof}
	If $\alg A \in \{\ZZ_1, \ZZ_2, \ZZ_4\}$, then the algebra $\alg A + \alg B_1 + \dots + \alg B_n + \ZZ_2$ is projective in $\HA$, and therefore projective in $\Sc$ as well.
	Since $\ZZ_9 \in \Sc$ and, by Proposition~\ref{pr-Z7h}, all algebras $\ZZ_n$ with $n \ge 10$ do not belong to $\Sc$, the algebra $\ZZ_9$ is the largest one-generated algebra in~$\Sc$. Hence, $\ZZ_9$ is free in~$\Sc$. 
	As $\ZZ_9 = \ZZ_8 + \ZZ_2$, we can apply Corollary~\ref{cor-prof}.
\end{proof}

\section{Well Quasi-Ordering.}\label{sec-qo}

In the proof of Theorem~\ref{th-fb}, we shall use the notion of well quasi-ordered sets and some of their properties (in particular, we shall employ Corollary~\ref{cor-fb}).

\begin{definition}[{see, e.g., \cite{Huczynska_Ruvskuc_Well_2015}}]
	A quasi-ordered set is said to be \Def{well-founded} if every strictly decreasing sequence is finite, and it is \Def{well quasi-ordered} if, in addition, it contains no infinite antichains.
\end{definition}

We next recall several auxiliary properties of well quasi-ordered sets.

\begin{proposition}[{see, e.g., \cite[Theorem~1.2]{Huczynska_Ruvskuc_Well_2015}}]\label{cor-qosc}
	The class of well quasi-ordered sets is closed under taking subsets and finite Cartesian products.
\end{proposition}

A subset of a quasi-ordered set whose elements are pairwise incomparable is called an \Def{antichain}.

\begin{proposition}\label{pr-qoantich}
	Let $(S_1, \leq_1)$ and $(S_2, \leq_2)$ be quasi-ordered sets. 
	If there exists a map $f: S_1 \to S_2$ such that, for any $a, b \in S_1$, 
	\[
	f(a) \leq_2 f(b) \implies a \leq_1 b,
	\]
	and if $S_2$ contains no infinite antichains, then $S_1$ also contains no infinite antichains.
\end{proposition}
\noindent
The proof is straightforward.

\begin{corollary}\label{cor-wqoweak}
	Let $(S_1, \leq_1)$ and $(S_2, \leq_2)$ be well-founded quasi-ordered sets. 
	If there exists a map $f: S_1 \to S_2$ such that, for any $a, b \in S_1$, 
	\[
	f(a) \leq_2 f(b) \implies a \leq_1 b,
	\]
	then $S_1$ is well quasi-ordered as well.
\end{corollary}

\begin{example}\label{ex-finqo}
	If $S$ is a set of finite s.i. algebras, it is quasi-ordered by $\leq$. In addition, $S$ is well quasi-ordered, because, if $\alg A$ and $\alg B$ are finite and $\alg A \leq \alg B$ with $\alg B \nleq \alg A$, the cardinality of $\alg B$ is less than cardinality of $\alg A$. Hence, $(S; \leq)$ is well quasi-ordered if and only if it does not contain infinite antichains. 
\end{example}

We shall also use the following well-known fact (see, e.g., \cite[Corollary~1.6]{Huczynska_Ruvskuc_Well_2015}) concerning well quasi-ordered sets.

\begin{proposition}\label{pr-dominaton}
	If $\lA$ is a well quasi-ordered alphabet, then the free monoid $\lA^*$ consisting of all words over the alphabet $\lA$ under the operation of concatenation is well quasi-ordered by the \emph{domination ordering}:
	\begin{align}
		\la_1 \dots \la_m \leq \lb_1 \dots \lb_n 
		\iff 
		(\exists\, 1 \leq j_1 < \dots < j_m \leq n)\,
		(\forall\, i = 1, \dots, m)\,(\la_i \leq \lb_{j_i}). 
		\label{eq-domination}
	\end{align}
\end{proposition}

\subsection{Quasi-Orders on Classes of Finite s.i. Algebras.}

Let $P$ be the class of all finite projective Heyting algebras. The goal of this section is to prove that $(P; \leq)$ is well quasi-ordered.

\begin{theorem}\label{th-prwqo}
	Let $P$ be the class of all finite projective Heyting algebras. Then $(P; \leq)$ is well quasi-ordered.
\end{theorem}

It follows from \cite{Balbes_Horn_1970} that $\alg A \in P$ if and only if $\alg A \cong \alg B_0 + \dots + \alg B_n + \ZZ_2$, where $\alg B_i \in \{\ZZ_2, \ZZ_4\}$. Since all algebras from $P$ are finite, by the argument from Example~\ref{ex-finqo}, it is sufficient to show that $P$ has no infinite antichains. Moreover, it suffices to show that $P$ does not contain an infinite antichain of algebras whose first summand is isomorphic to $\ZZ_4$.

Indeed, suppose that $\{\alg A_i\}_{i \ge 0}$ is an infinite antichain, and that $\alg A_i \cong \alg B_i + \alg C_i$, where $\alg B_i \in \{\ZZ_2, \ZZ_4\}$. It is clear that this antichain contains an infinite sub-antichain with the same $\alg B_i$. If $\alg B_i \cong \ZZ_4$, there is nothing to prove. If we have an infinite antichain of the form $\ZZ_2 + \alg C_{i_j}$, $j \ge 0$, we observe that $\ZZ_4 + \alg C_{i_j}$, $j \ge 0$, is also an antichain: indeed, if $\ZZ_4 + \alg C_{i_m} \leq \ZZ_4 + \alg C_{i_n}$, then a fortiori $\ZZ_2 + \alg C_{i_m} \leq \ZZ_2 + \alg C_{i_n}$.

Let $P'$ be a subset of $P$ consisting of all algebras of the form $\ZZ_4+\alg C$, 
and we need to show that $(P';\leq)$ is well-ordered.

We use the following abbreviation: for any $n \ge 0$,
\begin{align*}
	\An{n}:= \underbrace{\ZZ_4 + \dots + \ZZ_4 }_{{n} \text{ times}}+ \ZZ_2.
\end{align*}

Note that $\An{0} \cong \ZZ_2$ and 
thus, every algebra from $P'$ is of the form
\begin{align}
	\An{n_0} + \An{n_1} + \dots  + \An{n_m}
	\label{eq-dec}
\end{align}
where $n_0 > 0$ and for all $i \in [1,m]$, $n_i \ge 0$.

We will exploit the fact that algebras $\An{n}$ are "stretchable" over summands in the following sense.

\begin{proposition}[about stretching]\label{pr-stretch}
	Let $\alg A, \alg B$ and $\alg C$ be nontrivial Heyting algebras and $n \ge 0$ be a natural number. Then
	\begin{itemize}
		\item[(a)] if $\An{n} \leq \alg A$, then $\An{n} \leq \alg A + \alg B$;
		\item[(b)] if $\alg A \leq \alg A'$ and $\alg B \leq \alg B'$, then $\alg A + \An{0} + \alg B \leq \alg A' + \alg C + \alg B'$.
	\end{itemize} 
\end{proposition} 
\noindent The proof is straightforward.\\

Let $\lA = \{\la_n, n \ge 0\}$  be an alphabet, and with every algebra $\alg A^{(n)}$ we associate symbol $\la_n$;  more precisely, we define a map $f: \An{n} \mapsto \la_n$, for all $n \ge 0$.

Define on $\lA$ a quasi-order: 
\[ \la_i \leq \la_j \iff i \leq j. 
\]
It should be clear that $(\lA; \leq)$ is well quasi-ordered.

Like in Proposition~\ref{pr-dominaton}, let $\lA^*$ be a free monoid equipped with the domination quasi-order $\leq$, and let $\lA' = \{\la_n \mid n > 0\}$. As we know, $(\lA^*; \leq)$ is well quasi-ordered, as is $(\lA'; \leq)$, since $(\lA'; \leq)$ is a chain. Hence, by Proposition~\ref{cor-qosc}, the Cartesian product $\lA' \times \lA^*$ is well quasi-ordered by the componentwise quasi-order. We view the elements of $\lA' \times \lA^*$ as ordered pairs $\langle \la'; \la^* \rangle$, where $\la' \in \lA'$ and $\la^* \in \lA^*$. 

Now we extend the map $f$ to a map $f' \colon P' \to \lA' \times \lA^*$ defined by
\begin{align}
	f'(\An{n_0} + \An{n_1} + \dots + \An{n_m}) = (\la_{n_0}; \la_{n_0} \la_{n_1} \dots \la_{n_m}). \label{eq-map}
\end{align}

To complete the proof of Theorem~\ref{th-prwqo}, we will show that for every $\alg A, \alg B \in P'$,
\begin{align}
	\text{if } f'(\alg A) \leq f'(\alg B), \text{ then } \alg A \leq \alg B. \label{eq_AleqB}
\end{align}
Then, since the quasi-order on $P'$ is well-founded and $(\lA' \times \lA^*)$ is well quasi-ordered, one can apply Corollary~\ref{cor-wqoweak}.

Suppose that
\begin{align}
	&\alg A = \An{n_0} + \An{n_1} + \dots + \An{n_m}, \text{ and} \\
	&\alg B = \An{k_0} + \An{k_1} + \dots + \An{k_s}.
\end{align}

Then
\begin{align}
	&f'(\alg A) = (\la_{n_0}; \la_{n_0}\la_{n_1}\la_{n_2}\dots\la_{n_m}), \text{ and} \\
	&f'(\alg B) = (\la_{k_0}; \la_{k_0}\la_{k_1}\la_{k_2}\dots\la_{k_s}).
\end{align}

We prove~\eqref{eq_AleqB} by induction on $m$, the number of symbols in the second component of $f'(\alg A)$ minus~$1$.

\textit{Basis.} Suppose that $m = 1$. Then, by assumption, 
\[
(\la_{n_0}; \la_{n_0}) \leq (\la_{k_0}; \la_{k_0}\la_{k_1}\dots\la_{k_s}).
\]
By the definition of the quasi-order on the Cartesian product, $\la_{n_0} \leq \la_{k_0}$. Hence $\An{n_0} \leq \An{k_0}$, and by Proposition~\ref{pr-stretch}(a), $\An{n_0} \leq \alg B$.

\textit{Step.} Suppose that
\[
(\la_{n_0}; \la_{n_0}\la_{n_1}\la_{n_2}\dots\la_{n_m}, \la_{k_{m+1}}) \leq
(\la_{k_0}; \la_{k_0}\la_{k_1}\la_{k_2}\dots\la_{k_s}).
\]
Then, by the definition of the domination quasi-order, there exist $i < j$ such that $\la_{n_m} \leq \la_{k_i}$ and $\la_{n_{m+1}} \leq \la_{k_j}$. Hence, again by the definition of domination,
\[
(\la_{n_0}; \la_{n_0}\la_{n_1}\la_{n_2}\dots\la_{n_m}) \leq
(\la_{k_0}; \la_{k_0}\la_{k_1}\la_{k_2}\dots\la_{k_i}).
\]
By the induction hypothesis,
\[
\An{n_0} + \An{n_1} + \dots + \An{n_{m-1}} \leq
\An{k_0} + \An{k_1} + \dots + \An{k_i}.
\]
Therefore,
\[
\An{n_0} + \An{n_1} + \dots + \An{n_m} \leq
\An{k_0} + \An{k_1} + \dots + \An{k_j} + \dots + \An{k_s},
\]
that is, $\alg A \leq \alg B$. This completes the proof of Theorem~\ref{th-prwqo}.

\subsection{Well Quasi-Orders and Finite Bases.}

We use the following quasi-order on the class of all finite s.i.\ Heyting algebras. 

Let $\alg A$ and $\alg B$ be finite s.i.\ Heyting algebras. Then
\begin{align}
\alg A \preceq \alg B \iff \alg A \in \Var(\alg B). \label{eq-po}
\end{align}
Since the algebras are finite, by Jónsson’s Lemma,
\begin{align}
\alg A \preceq \alg B \iff \alg A \in \HH\CSub(\alg B). \label{eq-po1}
\end{align}

\begin{theorem}[{\cite[Lemma~1]{Citkin1986}}]\label{th-hla}
	Let $\VV$ be a locally finite variety of Heyting algebras, and let $\Lambda(\VV)$ be the complete lattice of all subvarieties of $\VV$. Let $S$ be the class of all finite s.i.\ algebras from $\VV$. Then the following are equivalent:
	\begin{itemize}
		\item[(a)] $\Lambda(\VV)$ has no infinite descending chains;
		\item[(b)] $S$ contains no infinite antichains with respect to $\preceq$;
		\item[(c)] the lattice $\Lambda(\VV)$ is countable. 
	\end{itemize}
\end{theorem}

\begin{remark}
	Theorem~\ref{th-hla} holds in a much broader setting; see~\cite[Theorem~5.4]{Citkin_Yankov_2022}.
\end{remark}

Let us note that property~(a) holds if and only if every subvariety of~$\VV$ is finitely based relative to~$\VV$. Thus, if $\VV$ itself is finitely based, condition~(a) is equivalent to the condition that every subvariety of~$\VV$ is finitely based. Moreover, since the algebras from~$S$ are finite, condition~(b) is equivalent to the condition that $S$ is well quasi-ordered.

\begin{corollary}\label{cor-fb}
	Let $\VV$ be a locally finite, finitely based variety of Heyting algebras, and let $S$ be the class of all finite s.i.\ algebras from~$\VV$. Then every subvariety of~$\VV$ is finitely based if and only if $(S, \leq)$ is well quasi-ordered. 
\end{corollary}

\section{Primitive Varieties of Heyting Algebras.}\label{sec-main}

The goal of this section is to establish the main results: the variety~$\WW$—the largest variety omitting all prohibited algebras—is the largest primitive variety of Heyting algebras, and  $\WW$ and all its subvarieties are finitely based.

\begin{note}
Throughout the remainder of the paper, $\WW$ denotes the largest variety of Heyting algebras omitting all prohibited algebras.
\end{note}

The main theorem can now be stated as follows.

\begin{theorem}[Main Theorem]\label{th-main1}
	A variety of Heyting algebras $\VV \subseteq \HA$ is primitive if and only if $\VV \subseteq \WW$.
\end{theorem}

The proof follows from the following lemmas.

\begin{lemma}[Necessary Condition]\label{lm-tnp}
	Any variety of Heyting algebras containing at least one of the prohibited algebras, is not primitive.
\end{lemma}

The converse follows from the observation that $\WW$ is primitive. To show that $\WW$ is primitive, we use the following three lemmas.

\begin{lemma}[{\cite[Proposition~5.1.24]{GorbunovBookE}}]\label{lm-lfpr}	
	A locally finite variety $\VV$ of finite type is primitive if and only if every finite s.i.\ algebra from $\VV$ is weakly $\VV$-projective.
\end{lemma}

\begin{lemma}\label{lm-lf}
	The variety $\WW$ is locally finite.
\end{lemma}		

\begin{lemma}\label{lm-pr}
	Every finite s.i.\ algebra from $\WW$ is weakly $\WW$-projective.
\end{lemma}		

Next, recall from~\cite{Jankov_1969} that for each finite s.i.\ Heyting algebra~$\alg A$, there exists an equation~$e_\alg A$ such that the subvariety defined by~$e_\alg A$ is the largest subvariety omitting the algebra~$\alg A$. Hence, the following holds.

\begin{theorem}\label{th-finbase}
	The variety $\WW$ is finitely based; namely, it can be axiomatized by the equations $e_{\PAlg_i}$, where $i \in [1,5]$.
\end{theorem}


\subsection{Proof of Lemma \ref{lm-tnp}}

By Corollary~\ref{cor-nonpr}, it suffices to demonstrate that all prohibited algebras, are totally non-projective. To that end, for each prohibited algebra $\PAlg_i$, we construct a homomorphic preimage $\PAlg_i^*$ in which $\PAlg_i$ cannot be embedded, while $\PAlg_i^* \in \Var(\PAlg_i)$ (cf. Fig.~\ref{fig-alg1_4},\ref{fig-alg5}). To show that $\PAlg_i^* \in \Var(\PAlg_i)$, we demonstrate that $\PAlg_i^*$ is a subdirect product of the algebras $\PAlg_i$ and $\ZZ_5$, and that $\ZZ_5 \in \CSub(\PAlg_i)$.

\begin{proof}
	We consider three cases.
	
	\textbf{Case $i = 1,3$.}
	First, observe (cf.~the Hasse diagrams in Fig.~\ref{fig-alg1_4}) that $\ZZ_5 \in \CSub(\PAlg_i)$, and hence $\ZZ_5 \in \Var(\PAlg_i)$.
	
	It is clear that $\PAlg_i$ is a subdirect product of the algebras $\PAlg_i/\pfltr{a}$ and $\PAlg_i/\pfltr{b}$. Observe that, as lattices, $\pidl{a}$ is isomorphic to $\ZZ_5$, that is, $\PAlg_i^*/\pfltr{a} \cong \ZZ_5$, and $\pidl{b}$ is isomorphic to $\PAlg_i$ itself, that is, $\PAlg_i^*/\pfltr{b} \cong \PAlg_i$. Thus, $\PAlg_i^*$ is (isomorphic to) a subdirect product of $\ZZ_5$ and $\PAlg_i$, and therefore $\PAlg_i^* \in \Var(\PAlg_i)$. We already know that $\PAlg_i \in \HH(\PAlg_i^*)$, and we leave it to the reader to verify that $\PAlg_i \notin \CSub(\PAlg_i^*)$. Hence, the algebras $\PAlg_i$ for $i = 1,3$ are totally non-projective.
	
	\textbf{Case $i = 2,4$.}
	Using the same argument with $\ZZ_5'$ instead of $\ZZ_5$, one can show that the algebras $\PAlg_2$ and $\PAlg_4$ are totally non-projective.
	
	\textbf{Case $i = 5$.}
	Consider the algebras whose Hasse diagrams are depicted in Fig.~\ref{fig-alg5}. First, observe that $\PAlg_5 \cong \alg A_5'$: the elements $a$ and $b$ generate the algebra $\PAlg_5$, and the map $f\colon a \mapsto a',\, b \mapsto b'$ can be lifted to an isomorphism from $\PAlg_5$ onto $\PAlg_5'$.
	
	Again, $\PAlg_5^*$ is a subdirect product of $\PAlg_5^*/\pfltr{a}$ and $\PAlg_5^*/\pfltr{b}$. It is easy to see that $\ZZ_5 \in \CSub(\PAlg_5)$ (cf.~Fig.~\ref{fig-alg5}, nodes in red).
	
	Observe that the ideal $\pidl{a}$ of the algebra $\PAlg_5^*$, as a lattice, is isomorphic to $\ZZ_5$, and hence $\alg A_5^*/\pfltr{a} \cong \alg \ZZ_5$; likewise, $\pidl{b}$, as a lattice, is isomorphic to $\PAlg_5'$ (cf.~Fig.~\ref{fig-alg5_1}, nodes in gray), and hence $\alg A_5^*/\pfltr{b} \cong \PAlg_5'$. Thus, $\PAlg_5^* \in \Var(\PAlg_5)$.
	
	We already know that $\PAlg_5 \in \HH(\PAlg_5^*)$, so to complete the proof that $\PAlg_5$ is totally non-projective, it remains to show that $\PAlg_5 \notin \CSub(\PAlg_5^*)$.
	
	Indeed, both algebras have exactly eight regular elements (cf.~Fig.~\ref{fig-alg5_1}, nodes in green), and every isomorphism preserves regularity. It is straightforward to verify that the regular elements of the algebra $\PAlg_5^*$ generate the entire algebra; therefore, $\PAlg_5 \notin \CSub(\PAlg_5^*)$.
\end{proof}

\begin{figure}[ht]
	\begin{tabular}{cccc}
		\begin{tikzpicture}[scale=1]
			\draw[fill] (0.5,1.5) -- (0.5,2);	
			\draw[fill] (0,0) -- (-0.5,0.5);	
			\draw[fill] (0,0) -- (1,1);	
			\draw[fill] (0,1) -- (0.5,0.5);	
			\draw[fill] (0.5,1.5) -- (-0.5,0.5);	
			\draw[fill] (0.5,1.5) -- (1,1);	
			\draw[fill,white] (0,0) -- (0,-0.5);	
			\draw[fill,red] (0,0) circle [radius=0.07];
			\draw[fill,red] (-0.5,0.5) circle [radius=0.07];
			\draw[fill] (0.5,0.5) circle [radius=0.07];
			\draw[fill] (0,1) circle [radius=0.07];				
			\draw[fill,red] (0.5,2) circle [radius=0.07];				
			\draw[fill,red] (1,1) circle [radius=0.07];				
			\draw[fill,red] (0.5,1.5) circle [radius=0.07];			
			\draw[fill,white] (0,-0.5) circle [radius=0.07];			
			\node[left]  at (1,2) {\footnotesize $\quad$};	
		\end{tikzpicture}	
		&	
		\begin{tikzpicture}[scale=1]
			\draw[fill] (0.5,1.5) -- (0.5,2);	
			\draw[fill] (0,0) -- (-0.5,0.5);	
			\draw[fill] (0,0) -- (1,1);	
			\draw[fill] (0,1) -- (0.5,0.5);	
			\draw[fill] (0.5,1.5) -- (-0.5,0.5);	
			\draw[fill] (0.5,1.5) -- (1,1);	
			\draw[fill] (0,0) -- (0,-0.5);	
			\draw[fill,red] (0,0) circle [radius=0.07];
			\draw[fill,red] (-0.5,0.5) circle [radius=0.07];
			\draw[fill] (0.5,0.5) circle [radius=0.07];
			\draw[fill] (0,1) circle [radius=0.07];				
			\draw[fill,red] (0.5,2) circle [radius=0.07];				
			\draw[fill,red] (1,1) circle [radius=0.07];				
			\draw[fill,red] (0.5,1.5) circle [radius=0.07];			
			\draw[fill,red] (0,-0.5) circle [radius=0.07];			
			\node[right]  at (1,2) {\footnotesize $\quad$};	
		\end{tikzpicture}	
		&		
		\begin{tikzpicture}[scale=1]
			\draw[fill] (0,1) -- (0,1.5);	
			\draw[fill] (-0.5,0.5) -- (0,1);	
			\draw[fill] (0.5,0.5) -- (0,1);	
			\draw[fill] (-0.5,0.5) -- (0,0);	
			\draw[fill] (0.5,0.5) -- (0,0);	
			\draw[fill] (-0.5,0) -- (0,0.5);	
			\draw[fill] (0.5,0) -- (0,0.5);	
			\draw[fill] (-0.5,0) -- (-0.5,0.5);	
			\draw[fill] (0.5,0) -- (0,-0.5);	
			\draw[fill] (-0.5,0) -- (0,-0.5);	
			\draw[fill] (0,0.5) -- (0,1);	
			\draw[fill] (0,-0.5) -- (0,0);	
			\draw[fill] (0.5,0) -- (0.5,0.5);	
			\draw[fill,white] (0,-1.1) -- (0,0);	
			\draw[fill] (0,0) -- (0,-0.5);	
			\draw[fill] (0,0) circle [radius=0.07];
			\draw[fill,red] (0,1) circle [radius=0.07];
			\draw[fill,red] (0,1.5) circle [radius=0.07];
			\draw[fill] (-0.5,0.5) circle [radius=0.07];
			\draw[fill,red] (0.5,0.5) circle [radius=0.07];
			\draw[fill] (0,0.5) circle [radius=0.07];	
			\draw[fill,red] (-0.5,0) circle [radius=0.07];	
			\draw[fill] (0.5,0) circle [radius=0.07];	
			\draw[fill,red] (0,-0.5) circle [radius=0.07];	
			\draw[fill,white] (0,-1) circle [radius=0.07];	
			\node[right]  at (1,0) {\footnotesize $ $};
		\end{tikzpicture}	
		&	
		\begin{tikzpicture}[scale=1]
			\draw[fill] (0,1) -- (0,1.5);	
			\draw[fill] (-0.5,0.5) -- (0,1);	
			\draw[fill] (0.5,0.5) -- (0,1);	
			\draw[fill] (-0.5,0.5) -- (0,0);	
			\draw[fill] (0.5,0.5) -- (0,0);	
			\draw[fill] (-0.5,0) -- (0,0.5);	
			\draw[fill] (0.5,0) -- (0,0.5);	
			\draw[fill] (-0.5,0) -- (-0.5,0.5);	
			\draw[fill] (0.5,0) -- (0,-0.5);	
			\draw[fill] (-0.5,0) -- (0,-0.5);	
			\draw[fill] (0,0.5) -- (0,1);	
			\draw[fill] (0,-0.5) -- (0,0);	
			\draw[fill] (0.5,0) -- (0.5,0.5);	
			\draw[fill] (0,-1) -- (0,0);	
			\draw[fill] (0,0) circle [radius=0.07];
			\draw[fill,red] (0,1) circle [radius=0.07];
			\draw[fill,red] (0,1.5) circle [radius=0.07];
			\draw[fill] (-0.5,0.5) circle [radius=0.07];
			\draw[fill,red] (0.5,0.5) circle [radius=0.07];
			\draw[fill] (0,0.5) circle [radius=0.07];	
			\draw[fill,red] (-0.5,0) circle [radius=0.07];	
			\draw[fill] (0.5,0) circle [radius=0.07];	
			\draw[fill,red] (0,-0.5) circle [radius=0.07];	
			\draw[fill,red] (0,-1) circle [radius=0.07];	
		\end{tikzpicture}
		\\
		$\PAlg_1$	& $\PAlg_2$&$\PAlg_3$	&$\PAlg_4$	
		\\
		& & & 
		\\
		\begin{tikzpicture}[scale=1] 
			\draw[fill] (-0.5,1.5) -- (0.5,2.5);	
			\draw[fill] (0,0) -- (-0.5,0.5);	
			\draw[fill] (0,0) -- (1,1);	
			\draw[fill] (-0.5,1.5) -- (0.5,0.5);	
			\draw[fill] (1,2) -- (-0.5,0.5);	
			\draw[fill] (0,2) -- (1,1);	
			\draw[fill,white] (0,0) -- (0,-0.5);	
			\draw[fill] (0.5,2.5) -- (1,2);	
			\draw[fill] (0,0) circle [radius=0.07];
			\draw[fill] (-0.5,0.5) circle [radius=0.07];
			\draw[fill] (0.5,0.5) circle [radius=0.07];
			\draw[fill] (0,1) circle [radius=0.07];				
			\draw[fill] (-0.5,1.5) circle [radius=0.07];				
			\draw[fill] (1,1) circle [radius=0.07];				
			\draw[fill] (0.5,1.5) circle [radius=0.07];			
			\draw[fill,white] (0,-0.5) circle [radius=0.07];			
			\draw[fill] (0,2) circle [radius=0.07];				
			\draw[fill] (0.5,2.5) circle [radius=0.07];				
			\draw[fill] (1,2) circle [radius=0.07];	
			\node[right]  at (1,2) {\footnotesize $b$};				
			\node[left]  at (-0.5,1.5) {\footnotesize $a$};	
		\end{tikzpicture}	
		&	
		\begin{tikzpicture}[scale=1] 
			\draw[fill] (-0.5,1.5) -- (0.5,2.5);	
			\draw[fill] (0,0) -- (-0.5,0.5);	
			\draw[fill] (0,0) -- (1,1);	
			\draw[fill] (-0.5,1.5) -- (0.5,0.5);	
			\draw[fill] (1,2) -- (-0.5,0.5);	
			\draw[fill] (0,2) -- (1,1);	
			\draw[fill] (0,0) -- (0,-0.5);	
			\draw[fill] (0.5,2.5) -- (1,2);	
			\draw[fill] (0,0) circle [radius=0.07];
			\draw[fill] (-0.5,0.5) circle [radius=0.07];
			\draw[fill] (0.5,0.5) circle [radius=0.07];
			\draw[fill] (0,1) circle [radius=0.07];				
			\draw[fill] (-0.5,1.5) circle [radius=0.07];				
			\draw[fill] (1,1) circle [radius=0.07];				
			\draw[fill] (0.5,1.5) circle [radius=0.07];			
			\draw[fill] (0,-0.5) circle [radius=0.07];			
			\draw[fill] (0,2) circle [radius=0.07];				
			\draw[fill] (0.5,2.5) circle [radius=0.07];				
			\draw[fill] (1,2) circle [radius=0.07];	
			\node[right]  at (1,2) {\footnotesize $b \quad$};				
			\node[left]  at (-0.5,1.5) {\footnotesize $a$};	
		\end{tikzpicture}	
		&		
		\begin{tikzpicture}[scale=1]
			\draw[fill] (0,1) -- (0,1.5);	
			\draw[fill] (-0.5,0.5) -- (0,1);	
			\draw[fill] (0.5,0.5) -- (0,1);	
			\draw[fill] (-0.5,0.5) -- (0,0);	
			\draw[fill] (0.5,0.5) -- (0,0);	
			\draw[fill] (-0.5,0) -- (0.5,1);	
			\draw[fill] (0.5,0) -- (0,0.5);	
			\draw[fill] (-0.5,0) -- (-0.5,0.5);	
			\draw[fill] (0.5,0) -- (0,-0.5);	
			\draw[fill] (-0.5,0) -- (0,-0.5);	
			\draw[fill] (0,0.5) -- (0,1);	
			\draw[fill] (0,-0.5) -- (0,0);	
			\draw[fill] (0.5,0) -- (0.5,0.5);	
			\draw[fill] (0,-0.5) -- (0,0);
			\draw[fill] (0.5,1) -- (0.5,2);	
			\draw[fill] (0,1) -- (0.5,1.5);	
			\draw[fill] (0,1.5) -- (0.5,2);	
			\draw[fill] (0,0) circle [radius=0.07];
			\draw[fill] (0,1) circle [radius=0.07];
			\draw[fill] (0,1.5) circle [radius=0.07];
			\draw[fill] (-0.5,0.5) circle [radius=0.07];
			\draw[fill] (0.5,0.5) circle [radius=0.07];
			\draw[fill] (0,0.5) circle [radius=0.07];	
			\draw[fill] (-0.5,0) circle [radius=0.07];	
			\draw[fill] (0.5,0) circle [radius=0.07];	
			\draw[fill] (0,-0.5) circle [radius=0.07];	
			\draw[fill,white] (0,-1) circle [radius=0.07];	
			\draw[fill] (0.5,2) circle [radius=0.07];	
			\draw[fill] (0.5,1.5) circle [radius=0.07];	
			\draw[fill] (0.5,1) circle [radius=0.07];	
			\node[right]  at (0.5,1) {\footnotesize $a$};
			\node[left]  at (0,1.5) {\footnotesize $b$};
		\end{tikzpicture}
		&	
		\begin{tikzpicture}[scale=1]
			\draw[fill] (0,1) -- (0,1.5);	
			\draw[fill] (-0.5,0.5) -- (0,1);	
			\draw[fill] (0.5,0.5) -- (0,1);	
			\draw[fill] (-0.5,0.5) -- (0,0);	
			\draw[fill] (0.5,0.5) -- (0,0);	
			\draw[fill] (-0.5,0) -- (0.5,1);	
			\draw[fill] (0.5,0) -- (0,0.5);	
			\draw[fill] (-0.5,0) -- (-0.5,0.5);	
			\draw[fill] (0.5,0) -- (0,-0.5);	
			\draw[fill] (-0.5,0) -- (0,-0.5);	
			\draw[fill] (0,0.5) -- (0,1);	
			\draw[fill] (0,-0.5) -- (0,0);	
			\draw[fill] (0.5,0) -- (0.5,0.5);	
			\draw[fill] (0,-1) -- (0,0);
			\draw[fill] (0.5,1) -- (0.5,2);	
			\draw[fill] (0,1) -- (0.5,1.5);	
			\draw[fill] (0,1.5) -- (0.5,2);	
			\draw[fill] (0,0) circle [radius=0.07];
			\draw[fill] (0,1) circle [radius=0.07];
			\draw[fill] (0,1.5) circle [radius=0.07];
			\draw[fill] (-0.5,0.5) circle [radius=0.07];
			\draw[fill] (0.5,0.5) circle [radius=0.07];
			\draw[fill] (0,0.5) circle [radius=0.07];	
			\draw[fill] (-0.5,0) circle [radius=0.07];	
			\draw[fill] (0.5,0) circle [radius=0.07];	
			\draw[fill] (0,-0.5) circle [radius=0.07];	
			\draw[fill] (0,-1) circle [radius=0.07];	
			\draw[fill] (0.5,2) circle [radius=0.07];	
			\draw[fill] (0.5,1.5) circle [radius=0.07];	
			\draw[fill] (0.5,1) circle [radius=0.07];	
			\node[right]  at (0.5,1) {\footnotesize $a$};
			\node[left]  at (0,1.5) {\footnotesize $b$};
		\end{tikzpicture}
		\\
		$\PAlg_1^*$	& $\PAlg_2^*$&$\PAlg_3^*$	&$\PAlg_4^*$\\
		&&&
		\\
		\begin{tikzpicture}[scale=1] 
			\draw[fill] (0,0) -- (-0.5,0.5);	
			\draw[fill] (0,0) -- (0.5,0.5);	
			\draw[fill] (0,1) -- (0.5,0.5);	
			\draw[fill] (0,1) -- (-0.5,0.5);	
			\draw[fill,white] (0,0) -- (0,-0.5);	
			\draw[fill] (0,1) -- (0,1.5);	
			\draw[fill] (0,0) circle [radius=0.07];
			\draw[fill] (-0.5,0.5) circle [radius=0.07];
			\draw[fill] (0.5,0.5) circle [radius=0.07];
			\draw[fill] (0,1) circle [radius=0.07];				
			\draw[fill] (0,1.5) circle [radius=0.07];	
		\end{tikzpicture}	& 
		\begin{tikzpicture}[scale=1] 
			\draw[fill] (0,0) -- (-0.5,0.5);	
			\draw[fill] (0,0) -- (0.5,0.5);	
			\draw[fill] (0,1) -- (0.5,0.5);	
			\draw[fill] (0,1) -- (-0.5,0.5);	
			\draw[fill] (0,0) -- (0,-0.5);	
			\draw[fill] (0,1) -- (0,1.5);	
			\draw[fill] (0,0) circle [radius=0.07];
			\draw[fill] (-0.5,0.5) circle [radius=0.07];
			\draw[fill] (0.5,0.5) circle [radius=0.07];
			\draw[fill] (0,1) circle [radius=0.07];				
			\draw[fill] (0,1.5) circle [radius=0.07];
			\draw[fill] (0,-0.5) circle [radius=0.07];	
		\end{tikzpicture}	& 	
		\\
		$\ZZ_5$ & $\ZZ_5'$&&
	\end{tabular} 
	\caption{Proof for $\alg A_1$ - $\alg A_4$}\label{fig-alg1_4}	
\end{figure}
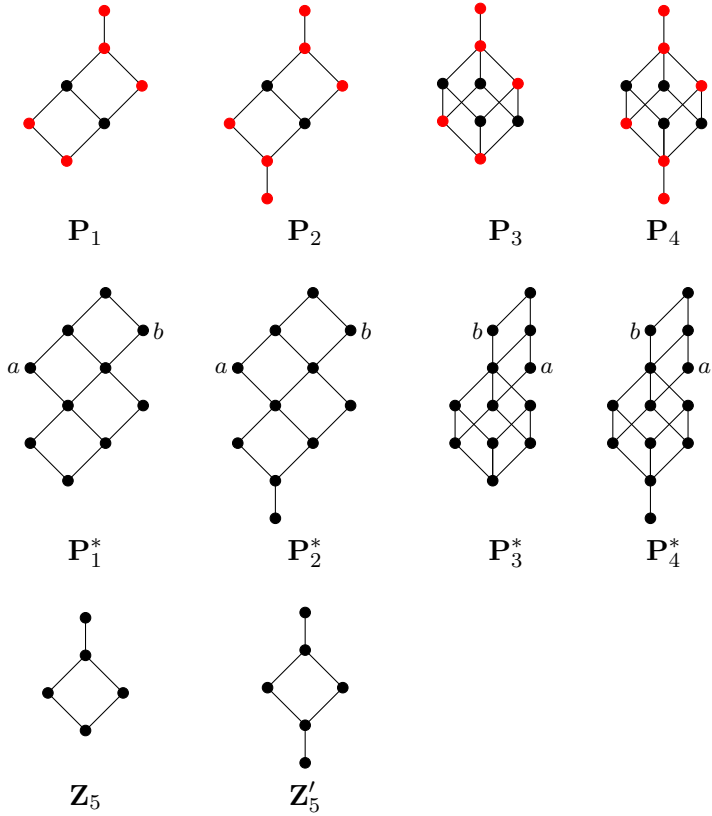

\begin{figure}[ht]
	\begin{tabular}{cccc}
		\begin{tikzpicture}[scale=1]
			\draw[fill] (0,2) -- (0,2.5);	
			\draw[fill] (-0.5,0.5) -- (0.5,1.5);	
			\draw[fill] (0.5,0.5) -- (-0.5,1.5);	
			\draw[fill] (-0.5,0.5) -- (0,0);	
			\draw[fill] (0.5,0.5) -- (0,0);	
			\draw[fill] (-0.5,0) -- (0,0.5);	
			\draw[fill] (0.5,0) -- (0,0.5);	
			\draw[fill] (-0.5,0) -- (-0.5,0.5);	
			\draw[fill] (0.5,0) -- (0,-0.5);	
			\draw[fill] (-0.5,0) -- (0,-0.5);	
			\draw[fill] (0,0.5) -- (0,1);	
			\draw[fill] (0,0) -- (1,1);	
			\draw[fill] (0,0) -- (-1,1);	
			\draw[fill] (0.5,0) -- (0.5,0.5);	
			\draw[fill] (0,2) -- (1,1);	
			\draw[fill] (0,2) -- (-1,1);	
			\draw[fill] (0,0) -- (0,-0.5);	
			\draw[fill,red] (0,0) circle [radius=0.07];
			\draw[fill,red] (0,1) circle [radius=0.07];
			\draw[fill] (0,2) circle [radius=0.07];
			\draw[fill] (-0.5,0.5) circle [radius=0.07];
			\draw[fill] (0.5,0.5) circle [radius=0.07];
			\draw[fill,red] (0,0.5) circle [radius=0.07];	
			\draw[fill] (-0.5,0) circle [radius=0.07];	
			\draw[fill] (0.5,0) circle [radius=0.07];	
			\draw[fill,red] (0,-0.5) circle [radius=0.07];	
			\draw[fill] (1,1) circle [radius=0.07];			
			\draw[fill] (-1,1) circle [radius=0.07];		
			\draw[fill] (-0.5,1.5) circle [radius=0.07];
			\draw[fill] (0.5,1.5) circle [radius=0.07];
			\draw[fill,red] (0,2.5) circle [radius=0.07];
			\node[left]  at (-0.5,0) {\footnotesize $a$};
			\node[left]  at (-1,1) {\footnotesize $\neg b$};
			\node[right]  at (0.5,0) {\footnotesize $b$};
			\node[right]  at (1,1) {\footnotesize $\neg a$};
		\end{tikzpicture}
		&
		\begin{tikzpicture}[scale=1]
			\draw[fill] (-0.5,0.5) -- (0.5,1.5);	
			\draw[fill] (0.5,0.5) -- (-0,1);	
			\draw[fill] (-0.5,0.5) -- (0,0);	
			\draw[fill] (0.5,0.5) -- (0,0);	
			\draw[fill] (-0.5,0) -- (0,0.5);	
			\draw[fill] (0.5,0) -- (0,0.5);	
			\draw[fill] (-0.5,0) -- (-0.5,0.5);	
			\draw[fill] (0.5,0) -- (0,-0.5);	
			\draw[fill] (-0.5,0) -- (0,-0.5);	
			\draw[fill] (0,0.5) -- (0,1.5);	
			\draw[fill] (0.5,0) -- (0.5,0.5);	
			\draw[fill] (0,0) -- (0,-0.5);	
			\draw[fill] (0,0.5) -- (0.5,1.2);	
			\draw[fill] (0,1.5) -- (0.5,2);	
			\draw[fill] (0.5,1.2) -- (0.5,2.5);	
			\draw[fill] (0.5,0.5) -- (0.5,0.8);	
			\draw[fill] (0,1.5) -- (0.5,0.8);	
			\draw[fill] (0,0) circle [radius=0.07];
			\draw[fill] (0,1) circle [radius=0.07];
			\draw[fill] (-0.5,0.5) circle [radius=0.07];
			\draw[fill] (0.5,0.5) circle [radius=0.07];
			\draw[fill] (0,0.5) circle [radius=0.07];	
			\draw[fill] (-0.5,0) circle [radius=0.07];	
			\draw[fill] (0.5,0) circle [radius=0.07];	
			\draw[fill] (0,-0.5) circle [radius=0.07];	
			\draw[fill] (0,1.5) circle [radius=0.07];	
			\draw[fill] (0.5,1.5) circle [radius=0.07];	
			\draw[fill] (0.5,1.2) circle [radius=0.07];	
			\draw[fill] (0.5,2) circle [radius=0.07];	
			\draw[fill] (0.5,2.5) circle [radius=0.07];	
			\draw[fill] (0.5,0.8) circle [radius=0.07];		\node[above]  at (0,0) {\footnotesize $a'$};
			\node[right]  at (0.5,1.2) {\footnotesize $\neg a'$};
			\node[left]  at (-0.5,0) {\footnotesize $b'$};
			\node[right]  at (0.5,0.8) {\footnotesize $\neg b'$};
		\end{tikzpicture}
		&
		\begin{tikzpicture}[scale=1]
			\draw[fill] (-0.5,0.5) -- (0.5,1.5);	
			\draw[fill] (0.5,0.5) -- (-0,1);	
			\draw[fill] (-0.5,0.5) -- (0,0);	
			\draw[fill] (0.5,0.5) -- (0,0);	
			\draw[fill] (-0.5,0) -- (0,0.5);	
			\draw[fill] (0.5,0) -- (0,0.5);	
			\draw[fill] (-0.5,0) -- (-0.5,0.5);	
			\draw[fill] (0.5,0) -- (0,-0.5);	
			\draw[fill] (-0.5,0) -- (0,-0.5);	
			\draw[fill] (0,0.5) -- (0,1.5);	
			\draw[fill] (0.5,0) -- (0.5,0.5);	
			\draw[fill] (0,0) -- (0,-0.5);	
			\draw[fill] (0,0.5) -- (0.5,1.2);	
			\draw[fill] (0,1.5) -- (0.5,2);	
			\draw[fill] (-1,1) -- (0,2);	
			\draw[fill] (-0.5,0.5) -- (-1,1);	
			\draw[fill] (0,1) -- (-0.5,1.5);	
			\draw[fill] (0.5,1.5) -- (0,2);	
			\draw[fill] (0.5,1.5) -- (0.5,2);	
			\draw[fill] (0.5,1.2) -- (0.5,1.5);	
			\draw[fill] (0.5,0.5) -- (0.5,0.8);	
			\draw[fill] (0,1.5) -- (0.5,0.8);	
			\draw[fill] (-0.5,2) -- (0.5,3);	
			\draw[fill] (0.5,2) -- (1,2.5);	
			\draw[fill] (0.5,3) -- (1,2.5);	
			\draw[fill] (0,2.5) -- (0.5,2);	
			\draw[fill] (-0.5,2) -- (0,1.5);	
			\draw[fill] (-0.5,1.5) -- (-0.5,2);	
			\draw[fill] (0,2) -- (0,2.5);	
			\draw[fill,lightgray] (0,0) circle [radius=0.07];
			\draw[fill,lightgray] (0,1) circle [radius=0.07];
			\draw[fill,lightgray] (-0.5,0.5) circle [radius=0.07];
			\draw[fill,lightgray] (0.5,0.5) circle [radius=0.07];
			\draw[fill,lightgray] (0,0.5) circle [radius=0.07];	
			\draw[fill,lightgray] (-0.5,0) circle [radius=0.07];	
			\draw[fill,lightgray] (0.5,0) circle [radius=0.07];	
			\draw[fill,lightgray] (0,-0.5) circle [radius=0.07];	
			\draw[fill,lightgray] (0,1.5) circle [radius=0.07];	
			\draw[fill,lightgray] (0.5,1.5) circle [radius=0.07];	
			\draw[fill,lightgray] (0.5,1.2) circle [radius=0.07];	
			\draw[fill,lightgray] (0.5,2) circle [radius=0.07];	
			\draw[fill] (-0.5,2) circle [radius=0.07];	
			\draw[fill] (-0,2.5) circle [radius=0.07];	
			\draw[fill,lightgray] (0.5,0.8) circle [radius=0.07];		
			\draw[fill] (-1,1) circle [radius=0.07];		
			\draw[fill] (-0.5,1.5) circle [radius=0.07];		
			\draw[fill] (0,2) circle [radius=0.07];		
			\draw[fill] (0.5,3) circle [radius=0.07];		
			\draw[fill,lightgray] (1,2.5) circle [radius=0.07];
			\node[left]  at (-1,1) {\footnotesize $a$};
			\node[right]  at (1,2.5) {\footnotesize $b$};
		\end{tikzpicture}
		
		\\
		$\PAlg_5$ & $\PAlg_5'$ & $\PAlg_5^*$ 
	\end{tabular}
	\caption{Proof for $\PAlg_5$}	\label{fig-alg5}
\end{figure}

\begin{figure}[ht]
	\begin{tabular}{cc}
		\begin{tikzpicture}[scale=1]
			\draw[fill] (0,2) -- (0,2.5);	
			\draw[fill] (-0.5,0.5) -- (0.5,1.5);	
			\draw[fill] (0.5,0.5) -- (-0.5,1.5);	
			\draw[fill] (-0.5,0.5) -- (0,0);	
			\draw[fill] (0.5,0.5) -- (0,0);	
			\draw[fill] (-0.5,0) -- (0,0.5);	
			\draw[fill] (0.5,0) -- (0,0.5);	
			\draw[fill] (-0.5,0) -- (-0.5,0.5);	
			\draw[fill] (0.5,0) -- (0,-0.5);	
			\draw[fill] (-0.5,0) -- (0,-0.5);	
			\draw[fill] (0,0.5) -- (0,1);	
			\draw[fill] (0,0) -- (1,1);	
			\draw[fill] (0,0) -- (-1,1);	
			\draw[fill] (0.5,0) -- (0.5,0.5);	
			\draw[fill] (0,2) -- (1,1);	
			\draw[fill] (0,2) -- (-1,1);	
			\draw[fill] (0,0) -- (0,-0.5);	
			\draw[fill,green] (0,0) circle [radius=0.07];
			\draw[fill] (0,1) circle [radius=0.07];
			\draw[fill] (0,2) circle [radius=0.07];
			\draw[fill] (-0.5,0.5) circle [radius=0.07];
			\draw[fill] (0.5,0.5) circle [radius=0.07];
			\draw[fill,green] (0,0.5) circle [radius=0.07];	
			\draw[fill,green] (-0.5,0) circle [radius=0.07];	
			\draw[fill,green] (0.5,0) circle [radius=0.07];	
			\draw[fill,green] (0,-0.5) circle [radius=0.07];	
			\draw[fill,green] (1,1) circle [radius=0.07];			
			\draw[fill,green] (-1,1) circle [radius=0.07];		
			\draw[fill] (-0.5,1.5) circle [radius=0.07];
			\draw[fill] (0.5,1.5) circle [radius=0.07];
			\draw[fill,green] (0,2.5) circle [radius=0.07];
			\node[left]  at (-0.5,0) {\footnotesize $a$};
			\node[left]  at (-1,1) {\footnotesize $\neg b$};
			\node[right]  at (0.5,0) {\footnotesize $b$};
			\node[right]  at (1,1) {\footnotesize $\neg a$};
		\end{tikzpicture}
		&
		\begin{tikzpicture}[scale=1]
			\draw[fill] (-0.5,0.5) -- (0.5,1.5);	
			\draw[fill] (0.5,0.5) -- (-0,1);	
			\draw[fill] (-0.5,0.5) -- (0,0);	
			\draw[fill] (0.5,0.5) -- (0,0);	
			\draw[fill] (-0.5,0) -- (0,0.5);	
			\draw[fill] (0.5,0) -- (0,0.5);	
			\draw[fill] (-0.5,0) -- (-0.5,0.5);	
			\draw[fill] (0.5,0) -- (0,-0.5);	
			\draw[fill] (-0.5,0) -- (0,-0.5);	
			\draw[fill] (0,0.5) -- (0,1.5);	
			\draw[fill] (0.5,0) -- (0.5,0.5);	
			\draw[fill] (0,0) -- (0,-0.5);	
			\draw[fill] (0,0.5) -- (0.5,1.2);	
			\draw[fill] (0,1.5) -- (0.5,2);	
			\draw[fill] (-1,1) -- (0,2);	
			\draw[fill] (-0.5,0.5) -- (-1,1);	
			\draw[fill] (0,1) -- (-0.5,1.5);	
			\draw[fill] (0.5,1.5) -- (0,2);	
			\draw[fill] (0.5,1.5) -- (0.5,2);	
			\draw[fill] (0.5,1.2) -- (0.5,1.5);	
			\draw[fill] (0.5,0.5) -- (0.5,0.8);	
			\draw[fill] (0,1.5) -- (0.5,0.8);	
			\draw[fill] (-0.5,2) -- (0.5,3);	
			\draw[fill] (0.5,2) -- (1,2.5);	
			\draw[fill] (0.5,3) -- (1,2.5);	
			\draw[fill] (0,2.5) -- (0.5,2);	
			\draw[fill] (-0.5,2) -- (0,1.5);	
			\draw[fill] (-0.5,1.5) -- (-0.5,2);	
			\draw[fill] (0,2) -- (0,2.5);	
			\draw[fill,green] (0,0) circle [radius=0.07];
			\draw[fill] (0,1) circle [radius=0.07];
			\draw[fill] (-0.5,0.5) circle [radius=0.07];
			\draw[fill] (0.5,0.5) circle [radius=0.07];
			\draw[fill] (0,0.5) circle [radius=0.07];	
			\draw[fill,green] (-0.5,0) circle [radius=0.07];	
			\draw[fill,green] (0.5,0) circle [radius=0.07];	
			\draw[fill,green] (0,-0.5) circle [radius=0.07];	
			\draw[fill] (0,1.5) circle [radius=0.07];	
			\draw[fill] (0.5,1.5) circle [radius=0.07];	
			\draw[fill,green] (0.5,1.2) circle [radius=0.07];	
			\draw[fill] (0.5,2) circle [radius=0.07];	
			\draw[fill] (-0.5,2) circle [radius=0.07];	
			\draw[fill] (-0,2.5) circle [radius=0.07];	
			\draw[fill,green] (0.5,0.8) circle [radius=0.07];		
			\draw[fill,green] (-1,1) circle [radius=0.07];		
			\draw[fill] (-0.5,1.5) circle [radius=0.07];		
			\draw[fill] (0,2) circle [radius=0.07];		
			\draw[fill,green] (0.5,3) circle [radius=0.07];		
			\draw[fill] (1,2.5) circle [radius=0.07];
			\node[left]  at (-1,1) {\footnotesize $\neg c$};
			\node[left]  at (-0.5,0) {\footnotesize $b$};
			\node[above]  at (0,0) {\footnotesize $a$};
			\node[right]  at (0.5,1.2) {\footnotesize $\neg a$};
			\node[right]  at (0.5,0.8) {\footnotesize $\neg b$};
			\node[right]  at (0.5,0) {\footnotesize $c$};
		\end{tikzpicture}
		\\
		$\PAlg_5$ & $\PAlg_5^*$
		
	\end{tabular}
	\caption{Proof for $\alg A_5$ (embedding)}	\label{fig-alg5_1}
\end{figure}

\subsection{Proof of Theorem \ref{th-lfa}.}

We need to prove that any variety of Heyting algebras omitting the algebra~$\PAlg_2$ is locally finite. In view of Theorem~\ref{th-lf}, it is sufficient to prove the following.

\begin{theorem}
	Let $\VV$ be a variety of Heyting algebras not containing the prohibited algebra~$\PAlg_2$ and satisfying the equation $\neg x \lor \neg\neg x \approx \one$. Then $\VV$ is locally finite.
\end{theorem}

To prove this statement, we will show that there exists a function~$\beta(n)$ bounding the cardinalities of $n$-generated s.i.\ algebras from~$\VV$, and then apply Theorem~\ref{th-crlf}. More precisely, by Corollary~\ref{cor-decomp}, for any $n$-generated nontrivial s.i.\ Heyting algebra~$\alg A \in \VV$, there exists a nodeless decomposition with $m$ components which are generated by at most $(2n + 2)^2$ elements.
 Since $\VV \models \neg x \lor \neg\neg x \approx \one$ and $\alg A$ is s.i., we have that $\two$ is the starting component.
We aim to show that all components are Boolean. As is well known, a $k$-generated Boolean algebra contains at most $2^{2^k}$ elements, which yields the desired bounding function~$\beta(n)$.

To that end, we prove the following lemma.

\begin{lemma}\label{lm-Bool}
	Let $\VV$ be a variety of Heyting algebras omitting the prohibited algebra~$\PAlg_2$. If $\alg B$ is a finitely generated, nontrivial, nodeless algebra and $\two + \alg B + \two \in \VV$, then $\alg B$ is Boolean.
\end{lemma}	

\begin{proof}
	To prove Lemma~\ref{lm-Bool}, we establish its contrapositive: if $\alg B$ is a finitely generated, nodeless, non-Boolean Heyting algebra, then $\two + \alg B + \two \notin \VV$.
	
	Indeed, by Corollary~\ref{cor-ord}, if the algebra~$\alg B$ is nodeless and non-Boolean, it contains an ordinary element~$g$. This element generates a cyclic subalgebra~$\ZZ_n$ with $n \geq 6$. We aim to show that the algebra~$\two + \ZZ_7$, that is, the algebra~$\PAlg_2$, always belongs to~$\VV$.
	
	If $n = \infty$, then, since the algebra~$\ZZ_7$ is a homomorphic image of~$\ZZ_\infty$, the algebra~$\two + \ZZ_7$ is a homomorphic image of $\two + \ZZ_\infty + \two$, and therefore $\PAlg_2 \cong \two + \ZZ_7 \in \VV$.
	
	If $6 \leq n < \infty$, then, since the algebra~$\two + \ZZ_7$ is a subalgebra of $\two + \ZZ_n + \two$, it follows that $\PAlg_2 \cong \two + \ZZ_7 \in \VV$.
\end{proof}

This observation completes the proof of Theorem~\ref{th-lfa}.

\subsection{Proof of Lemma \ref{lm-pr}.}

To prove Lemma \ref{lm-pr} it suffices to prove the following.

\begin{lemma}\label{lm-Boole}
	For any $\alg A \in \WW$, if
	\begin{align}
		\alg A = \sum_{i=0}^m \alg B_i + \two,\label{eq-hscpl}
	\end{align}
	where $\alg B_i$ are nodeless algebras, then $\alg B_0 \in \{\ZZ_1,\ZZ_2,\ZZ_4,\ZZ_8\}$ and $\alg B_i \in \{\ZZ_2,\ZZ_4\}$ for all $i \in [1,m]$.
\end{lemma}

\subsubsection{Proof of Lemma \ref{lm-Boole}.}

First, observe that since $\WW$ omits the prohibited algebra $\PAlg_2$, it is locally finite, and any finitely generated algebra from $\WW$ is finite. Hence, every finitely generated s.i.\ algebra $\alg A$ from $\WW$ has a nodeless decompositions  of the form \eqref{lm-Boole} 
where all algebras $\alg B_i$, $i \in [1,m]$, are Boolean. Moreover, since $\WW$ omits $\PAlg_4$, we have that for each $i \in [1,m]$, $\alg B_i \in \{\ZZ_2, \ZZ_4\}$. Furthermore, because $\WW$ omits $\PAlg_1$ (which is $\ZZ_7$), by Corollary~\ref{cor-Scott}, if $\alg B_0 \in \{\ZZ_1, \ZZ_2, \ZZ_4, \ZZ_8\}$, then the algebra $\alg A$ is $\WW$-projective. Therefore, it suffices to prove the following statement.

\begin{lemma}\label{pr-densebool}
	Let $\VV$ be a subvariety of $\WW$. Then, for any finite nodeless algebra $\alg B$ such that $\alg B + \two \in \VV$, we have $\alg B \in \{\ZZ_1, \ZZ_2, \ZZ_4, \ZZ_8\}$.
\end{lemma}

If $\alg B$ is trivial, the statement is immediate. Therefore, we assume that $\alg B$ is nontrivial.

First, we prove the following.

\begin{claim}\label{cl-611}
	Let $d \in \alg B$ be the least dense element. Then the algebra $\alg B\pfltr{d}$ belongs to $\{\ZZ_1, \ZZ_2, \ZZ_4\}$.
\end{claim}

Suppose that $d$ is the least dense element of $\alg B$. By Proposition~\ref{pr-nodeZ7}, since $\VV$ omits $\PAlg_2$ (which is isomorphic to $\ZZ_2 + \ZZ_7$), the algebra $\alg B\pfltr{d}$ is nodeless. Moreover, because $\two + \alg B\pfltr{d} + \two$ is a subalgebra of $\alg B + \two$ and hence $\two + \alg B\pfltr{d} + \two \in \VV$, we can apply Lemma~\ref{lm-Bool} and conclude that $\alg B\pfltr{d}$ is a Boolean algebra containing at most four elements, since $\PAlg_4 \notin \VV$.

Next, we consider three cases:
\begin{enumerate}
	\item $\alg B$ contains a single dense element;
	\item $\alg B$ contains two dense elements;
	\item $\alg B$ contains four dense elements. 
\end{enumerate}

\textbf{Case 1.} If $\alg B$ contains a single dense element, then by Proposition~\ref{pr-ndldense}(a), $\alg B$ is Boolean. As $\PAlg_3 \notin \VV$, it follows that $\alg B \in \{\ZZ_1, \ZZ_2, \ZZ_4\}$.

\textbf{Case 2.} Suppose that $\alg B$ contains exactly two dense elements. Since it is nodeless, by Proposition~\ref{pr-ndldense}(b) it contains a subalgebra isomorphic to $\ZZ_6$. Hence, by Proposition~\ref{pr-sumsub}, the algebra $\alg B + \two$ contains a subalgebra $\ZZ_6 + \two$, which is isomorphic to $\PAlg_1$, while $\PAlg_1 \notin \WW$ . Therefore, $\alg B + \two \notin \VV$, contrary to the assumption. Thus, Case~2 is impossible.

\textbf{Case 3.} Assume that $\alg B\pfltr{d}$ is a four-element Boolean algebra. Then, by Proposition~\ref{pr-ndldense}(c), $\alg B$ is a subalgebra of a Cartesian product of some algebras $\alg C_1 + \two$ and $\alg C_2 + \two$, where $\alg C_i$, $i = 1, 2$, are Boolean algebras. The algebras $\alg C_i$ are homomorphic images of $\alg B$, which in turn is a homomorphic image of $\alg B + \two$, and the latter algebra belongs to $\VV$. Hence, $\alg C_i \in \VV$. Since $\PAlg_3 \notin \VV$, the cardinalities of the algebras $\alg C_i$ are at most $4$; that is,
\[
\alg C_i + \two \in \{\ZZ_1 + \two,\, \ZZ_2 + \two,\, \ZZ_4 + \two\} = \{\ZZ_2,\, \ZZ_3,\, \ZZ_5\}.
\]
Noting that the algebras $\ZZ_2$ and $\ZZ_3$ are subalgebras of $\ZZ_5$, we may, without loss of generality, assume that $\alg B \in \CSub(\ZZ_5^2)$.

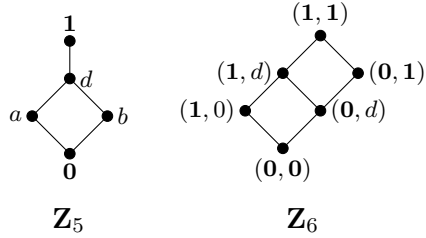
\begin{figure}[ht]
	\centering
	\begin{tabular}{cc}
		\begin{tikzpicture}[scale=1] 
			\draw[fill] (0,0) -- (-0.5,0.5);	
			\draw[fill] (0,0) -- (0.5,0.5);	
			\draw[fill] (0,1) -- (0.5,0.5);	
			\draw[fill] (0,1) -- (-0.5,0.5);	
			\draw[fill,white] (0,0) -- (0,-0.5);	
			\draw[fill] (0,1) -- (0,1.5);	
			\draw[fill] (0,0) circle [radius=0.07];
			\draw[fill] (-0.5,0.5) circle [radius=0.07];
			\draw[fill] (0.5,0.5) circle [radius=0.07];
			\draw[fill] (0,1) circle [radius=0.07];				
			\draw[fill] (0,1.5) circle [radius=0.07];	
			\node[below]  at (0,0) {\footnotesize $\zero$};	
			\node[left]  at (-0.5,0.5) {\footnotesize $a$};	
			\node[right]  at (0.5,0.5) {\footnotesize $b$};	
			\node[right]  at (0,1) {\footnotesize $d$};	
			\node[above]  at (0,1.5) {\footnotesize $\one$};	
		\end{tikzpicture}	
		&
		\begin{tikzpicture}[scale=1]
			\draw[fill] (0,0) -- (-0.5,0.5);	
			\draw[fill] (0,0) -- (1,1);	
			\draw[fill] (0,1) -- (0.5,0.5);	
			\draw[fill] (0.5,1.5) -- (-0.5,0.5);	
			\draw[fill] (0.5,1.5) -- (1,1);	
			\draw[fill,white] (0,0) -- (0,-0.5);	
			\draw[fill] (0,0) circle [radius=0.07];
			\draw[fill] (-0.5,0.5) circle [radius=0.07];
			\draw[fill] (0.5,0.5) circle [radius=0.07];
			\draw[fill] (0,1) circle [radius=0.07];				
			\draw[fill] (1,1) circle [radius=0.07];				
			\draw[fill] (0.5,1.5) circle [radius=0.07];			
			\node[below]  at (0,0) {\footnotesize $(\zero,\zero)$};
			\node[right]  at (0.5,0.5) {\footnotesize $(\zero,d)$};	
			\node[left]  at (-0.5,0.5) {\footnotesize $(\one,0)$};	
			\node[left]  at (0,1) {\footnotesize $(\one,d)$};
			\node[right]  at (1,1) {\footnotesize $(\zero,\one)$};	
			\node[above]  at (0.5,1.5) {\footnotesize $(\one,\one)$};	
			
		\end{tikzpicture}	 
		\\
		$\ZZ_5$& $\ZZ_6$
	\end{tabular}
	\caption{Proof of Lemma \ref{pr-densebool}}\label{fig-Z5}
\end{figure}

Let us take a closer look at the subalgebras of the algebra $\ZZ_5^2$, whose elements are the ordered pairs $(x, y)$, where $x, y \in \{\zero, a, b, d, \one\}$ (see Fig.~\ref{fig-Z5}). 

Suppose that $\alg C$ is a subalgebra of $\alg B$. If $\alg C$ contains fewer than four dense elements, we can repeat the arguments used in Cases~1 and~2. Thus, assume that $\alg C$ contains four dense elements. Since the algebra $\ZZ_5^2$ contains precisely four dense elements, namely $(d, d)$, $(d, \one)$, $(\one, d)$, and $(\one, \one)$, all of them must belong to $\alg C$.

Next, note that each of the elements $(\zero, d)$ and $(d, \zero)$ generates a subalgebra isomorphic to $\ZZ_6$ (see Fig.~\ref{fig-Z5}). Hence, if either of these elements were in $\alg C$, the algebra $\alg B + \two$ would contain a subalgebra isomorphic to $\ZZ_6 + \two$, and consequently we would have $\PAlg_1 \cong \ZZ_6 + \two \in \VV$, contradicting the assumption. Therefore, $(\zero, d), (d, \zero) \notin \alg C$. 

Since $(\zero, \one) \land (d, d) = (\zero, d)$ and the element $(d, d)$ is always in $\alg C$, the element $(\zero, \one)$ cannot belong to $\alg C$, and by a similar argument, $(\one, \zero) \notin \alg C$. 

Furthermore, because $(\zero, a) \lor (\zero, b) = (\zero, d)$, the elements $(\zero, a)$ and $(\zero, b)$ cannot simultaneously belong to $\alg C$. Thus, there are nine possible configurations for the membership of the elements $(\zero, a)$, $(\zero, b)$, $(a, \zero)$, and $(b, \zero)$ in $\alg C$:

\begin{table}[ht]
	\centering
	\begin{tabular}{c|ccccccccc}
		&1&2&3&4&5&6&7&8&9\\
		\hline
		$(\zero, a)$ & $-$ & $+$ & $-$ & $-$ & $-$ & $+$ & $+$ & $-$ & $-$ \\ 
		$(\zero, b)$ & $-$ & $-$ & $-$ & $-$ & $-$ & $-$ & $-$ & $+$ & $+$ \\ 
		$(a, \zero)$ & $-$ & $-$ & $-$ & $+$ & $-$ & $+$ & $-$ & $+$ & $-$ \\ 
		$(b, \zero)$ & $-$ & $-$ & $-$ & $-$ & $+$ & $-$ & $+$ & $-$ & $+$ \\ 
	\end{tabular}
	\caption{Possible configurations of subalgebras $\alg C$}\label{tbl-salg}
\end{table}

Because any one-to-one map of $\{(\zero, a), (\zero, b), (a, \zero), (b, \zero)\}$ can be extended to an automorphism of $\ZZ_5^2$, Cases~2--5 can be considered similarly, as can Cases~6--9. Thus, it suffices to analyze only Cases~1,~2, and~6.

1. This case is impossible, because if the elements $(\zero, a)$, $(\zero, b)$, $(a, \zero)$, and $(b, \zero)$ are not in $\alg C$, then the elements $(\one, b)$, $(\one, a)$, $(b, \one)$, and $(a, \one)$ also cannot belong to $\alg C$. Consequently, the element $(d, d)$ would be a nontrivial node, contrary to the assumption.

2. Assume that $(\zero, a) \in \alg C$ and that $(\zero, b)$, $(a, \zero)$, and $(b, \zero)$ are not in $\alg C$. Then $\alg C \cong \ZZ_8$: the element $(\zero, a)$ generates a subalgebra isomorphic to $\ZZ_8$ (see Fig.~\ref{fig-case3}), and, as shown in Table~\ref{tbl-ZZ52}, the remaining elements of $\ZZ_5^2$ cannot belong to $\alg C$. For instance, the element $(a, a)$ is excluded (row~7), because the element $(a, \zero)$—which is excluded by assumption—can be expressed in terms of $(a, a)$ and $(\zero, a)$, the latter being an element of $\alg C$.

\begin{figure}[ht]
	\begin{tabular}{ccc}
		
		\begin{tikzpicture}[scale=1]
			\draw[fill] (0,1) -- (0,1.5);	
			\draw[fill] (-0.5,0.5) -- (0,1);	
			\draw[fill] (0.5,0.5) -- (0,1);	
			\draw[fill] (-0.5,0.5) -- (0,0);	
			\draw[fill] (0.5,0.5) -- (0,0);	
			\draw[fill] (-0.5,0) -- (0,0.5);	
			\draw[fill] (0.5,0) -- (0,0.5);	
			\draw[fill] (-0.5,0) -- (-0.5,0.5);	
			\draw[fill] (0.5,0) -- (0,-0.5);	
			\draw[fill] (-0.5,0) -- (0,-0.5);	
			\draw[fill] (0,0.5) -- (0,1);	
			\draw[fill] (0,-0.5) -- (0,0);	
			\draw[fill] (0.5,0) -- (0.5,0.5);	
			\draw[fill,white] (0,-1.1) -- (0,0);	
			\draw[fill] (0,0) -- (0,-0.5);	
			\draw[fill] (0,0) circle [radius=0.07];
			\draw[fill] (0,1) circle [radius=0.07];
			\draw[fill] (0,1.5) circle [radius=0.07];
			\draw[fill] (-0.5,0.5) circle [radius=0.07];
			\draw[fill] (0.5,0.5) circle [radius=0.07];
			\draw[fill] (0,0.5) circle [radius=0.07];	
			\draw[fill] (-0.5,0) circle [radius=0.07];	
			\draw[fill] (0.5,0) circle [radius=0.07];	
			\draw[fill] (0,-0.5) circle [radius=0.07];	
			\draw[fill,white] (0,-1) circle [radius=0.07];	
			\node[below]  at (0,-0.5) {\footnotesize $(\zero,\zero) $};
			\node[right]  at (0.5,0) {\footnotesize $(\zero,a) $};
			\node[left]  at (-0.5,0.5) {\footnotesize $(\one,b) $};
			\node[right]  at (0.5,0.5) {\footnotesize $(a,a) $};
			\node[left]  at (-0.5,0) {\footnotesize $(b,b) $};
		\end{tikzpicture}
		&
		\begin{tikzpicture}[scale=1] 
			\draw[fill] (0,0) -- (-0.5,0.5);	
			\draw[fill] (0,0) -- (1,1);	
			\draw[fill] (-0.5,1.5) -- (0.5,0.5);	
			\draw[fill] (0.5,1.5) -- (-0.5,0.5);	
			\draw[fill] (1,1) -- (0,2);	
			\draw[fill] (-0.5,1.5) -- (0,2);	
			\draw[fill] (0,0) circle [radius=0.07];
			\draw[fill] (-0.5,0.5) circle [radius=0.07];
			\draw[fill] (0.5,0.5) circle [radius=0.07];
			\draw[fill] (0,1) circle [radius=0.07];				
			\draw[fill] (0.5,1.5) circle [radius=0.07];
			\draw[fill] (1,1) circle [radius=0.07];
			\draw[fill] (-0.5,1.5) circle [radius=0.07];
			\draw[fill] (0,2) circle [radius=0.07];
			\draw[fill,white] (0,-0.5) circle [radius=0.07];
			\node[below]  at (0,0) {\footnotesize $(\zero,\zero)$};
			\node[right]  at (0.5,0.5) {\footnotesize $(d,b)$};
			\node[left]  at (-0.5,0.5) {\footnotesize $(\zero,a)$};	
			\node[left]  at (0,1) {\footnotesize $(d,d)$};
			\node[right]  at (1,1) {\footnotesize $(\one,b)$};
			\node[right]  at (0.5,1.5) {\footnotesize $(\one,d)$};
			\node[left]  at (-0.5,1.5) {\footnotesize $(d,\one)$};
			\node[above]  at (0,2) {\footnotesize $(\one,\one)$};	
		\end{tikzpicture}&
		\begin{tikzpicture}[scale=1]
			\draw[fill] (-0.5,0.5) -- (0.5,1.5);	
			\draw[fill] (0.5,0.5) -- (-0.5,1.5);	
			\draw[fill] (-0.5,0.5) -- (0,0);	
			\draw[fill] (0.5,0.5) -- (0,0);	
			\draw[fill] (-0.5,0) -- (0,0.5);	
			\draw[fill] (0.5,0) -- (0,0.5);	
			\draw[fill] (-0.5,0) -- (-0.5,0.5);	
			\draw[fill] (0.5,0) -- (0,-0.5);	
			\draw[fill] (-0.5,0) -- (0,-0.5);	
			\draw[fill] (0,0.5) -- (0,1);	
			\draw[fill] (0,0) -- (1,1);	
			\draw[fill] (0,0) -- (-1,1);	
			\draw[fill] (0.5,0) -- (0.5,0.5);	
			\draw[fill] (0,2) -- (1,1);	
			\draw[fill] (0,2) -- (-1,1);	
			\draw[fill] (0,0) -- (0,-0.5);	
			\draw[fill] (0,0) circle [radius=0.07];
			\draw[fill] (0,1) circle [radius=0.07];
			\draw[fill] (0,2) circle [radius=0.07];
			\draw[fill] (-0.5,0.5) circle [radius=0.07];
			\draw[fill] (0.5,0.5) circle [radius=0.07];
			\draw[fill] (0,0.5) circle [radius=0.07];	
			\draw[fill] (-0.5,0) circle [radius=0.07];	
			\draw[fill] (0.5,0) circle [radius=0.07];	
			\draw[fill] (0,-0.5) circle [radius=0.07];	
			\draw[fill] (1,1) circle [radius=0.07];			
			\draw[fill] (-1,1) circle [radius=0.07];		
			\draw[fill] (-0.5,1.5) circle [radius=0.07];
			\draw[fill] (0.5,1.5) circle [radius=0.07];
			\node[below]  at (0,-0.5) {\footnotesize $(\zero,\zero) $};
			\node[right]  at (0.5,0) {\footnotesize $(a,\zero) $};
			\node[left]  at (-0.5,0) {\footnotesize $(\zero,a) $};
			\node[right]  at (0.5,0.5) {\footnotesize $(d,b) $};
			\node[left]  at (-0.5,0.5) {\footnotesize $(b,d) $};
			\node[right]  at (1,1) {\footnotesize $(\one,b) $};
			\node[left]  at (-1,1) {\footnotesize $(b,\one) $};
			\node[right]  at (0.5,1.5) {\footnotesize $(\one,d) $};
			\node[left]  at (-0.5,1.5) {\footnotesize $(d,\one) $};
			\node[above]  at (0,2) {\footnotesize $(\one,\one) $};
			\node[above]  at (0,1) {\footnotesize $(d,d) $};
		\end{tikzpicture}	
		\\
		$\ZZ_5$&$\ZZ_8$&$\alg D$\\
		(a)&(b)&(c)
	\end{tabular}
	\caption{Case 3}\label{fig-case3}
\end{figure}

\begin{table}[ht]
	\begin{tabular}{l|c|c|l}
		&Element & Included Y/N & Reason
		\\
		\hline
		1& $(\zero,\zero)$ & $+$ &Fig. \ref{fig-case3}(b)\\
		2& $(\zero,a)$ & + & Fig. \ref{fig-case3}(b)\\
		3& $(\zero,b)$ & $-$ & Excluded by Case 2\\
		4& $(\zero,d)$ & $-$ & prohibited element \\
		5& $(\zero,\one)$ & $-$ & prohibited element \\
		6& $(a,\zero)$ & $-$ & Excluded by Case 2\\
		7& $(a,a)$&  $- $ &  $(a,\zero)= \neg(\neg (a,a) \lor (\zero,a))$  \\
		8& $(a,b)$ & $-$ &  $(b,\zero) = \neg((a,b) \lor (\zero,a))$\\
		9& $(a,d)$ & $-$ & $(b,\zero) =  \neg(a,d)$\\
		10& $(a,\one)$ & $-$ & $(b,\zero) = \neg(a,\one)$\\
		11& $(b,\zero)$ & $-$ & Excluded by Case 2\\
		12& $(b,a)$  & $-$ &$(b, \zero) = \neg(\neg(b,a) \lor (\zero,a))$\\
		13& $(b,b)$ & $- $&$(a,\zero) = \neg((b,b) \lor (\zero,a))$\\
		14& $(b,d)$ & $- $& $(a,\zero) = \neg(b,d)$\\
		15& $(b,\one)$ & $- $ &$(a,\zero) = \neg(b,\one)$\\
		16& $(d,\zero)$ & $-$ & prohibited element \\
		17& $(d,a)$ & $- $ & $(\zero,b) = \neg(d,a) $\\
		18& $(d,b)$ & $+$ & Fig. \ref{fig-case3}(b)\\
		19& $(d,d)$ & $+$ & Fig. \ref{fig-case3}(b)\\
		20& $(d,\one)$ & $+$ & Fig. \ref{fig-case3}(b)\\
		21& $(\one,\zero)$ & $-$ &prohibited element\\
		22& $(\one,a)$ & $-$ &$(\zero,b) = \neg(\one,a)$\\
		23& $(\one,b)$ & $+$ & Fig. \ref{fig-case3}(b)\\
		24& $(\one,d)$ & $+$ & Fig. \ref{fig-case3}(b)\\
		25& $(\one,\one)$& $+$  & Fig. \ref{fig-case3}(b)\\
		
	\end{tabular}
	\caption{Subalgebras of $\ZZ_5^2$}\label{tbl-ZZ52}
\end{table}

3. The case when $(\zero, b), (a, \zero) \in \alg C$ is impossible, because these two elements, together with the element $(d, d)$ (which is in $\alg C$ by assumption), generate the algebra $\alg D$ (see Fig.~\ref{fig-case3}). Consequently, the algebra $\PAlg_5$ is a subalgebra of $\alg B + \two \in \VV$, while $\alg A_5 \notin \VV$ by assumption. This observation completes the proof of the claim. \\

Now, we are in a position to complete the proof of Lemma~\ref{lm-pr}: it remains to show that each finite s.i. algebra from the variety $\WW$ is weakly $\WW$-projective. In fact, these algebras are $\WW$-projective.  

Indeed, by definition, $\WW$ is a subvariety of the Scott variety $\Sc$, and by Corollary~\ref{cor-Scott}, all algebras in $\Sc$ of the form \eqref{eq-hscpl} are $\Sc$-projective. Moreover, by Lemma~\ref{lm-Boole}, all finitely generated s.i. algebras from $\WW$ are of the form \eqref{eq-hscpl}.

\section{Finite Bases.}\label{sec-fb}
	
	In this section we present a proof of Theorem \ref{th-fb}. \\
	
To prove that all subvarieties of the variety~$\WW$ are finitely based, we recall (cf. Theorem~\ref{th-finbase}) that the variety~$\WW$ itself is finitely based and, by Theorem~\ref{th-lfa}, it is locally finite. 
Thus, let $S$ be the class of all finite s.i.\ members of~$\WW$. If we can demonstrate that $S$ is well quasi-ordered with respect to~$\preceq$ (defined in~\eqref{eq-po1}), then, by Corollary~\ref{cor-fb}, we may conclude that all subvarieties of~$\WW$ are finitely based. 

By Lemma~\ref{lm-pr}, all members of~$S$ are weakly $\WW$-projective. Therefore, for any $\alg A, \alg B \in S$, 
\[
\alg A \preceq \alg B \iff \alg A \in \HH\CSub(\alg B) \iff \alg A \in \CSub(\alg B) \iff \alg A \leq \alg B.
\]
We aim to show that $(S;\leq)$ is well quasi-ordered.

\begin{lemma}\label{lem_redpr}
	If $(S, \leq)$ contains an infinite antichain, then it contains an infinite antichain of algebras of the form
	\[
	\alg A = \sum_{i=0}^m \alg B_i + \two,
	\]
	where $\alg B_i \in \{\ZZ_2, \ZZ_4\}$ for all $i \in [0, m]$. That is, $(S, \leq)$ contains an infinite antichain of finite projective Heyting algebras. 	
\end{lemma}

\begin{proof}
Indeed, if $(S, \leq)$ contains an infinite antichain, then it contains an infinite subantichain of algebras having the same first summand. That is, $(S, \leq)$ contains an infinite antichain of algebras in each of which $\alg B_0$ is either $\ZZ_2$, $\ZZ_4$, or $\ZZ_8$. In the first two cases there is nothing to prove. Let us consider the remaining case.

Suppose that the algebras $\ZZ_8 + \alg C_j$, $j \ge 0$, form an antichain. Then, since $\ZZ_8 \inc \ZZ_4$, by Corollary~\ref{cor_suminc}(b), the algebras $\ZZ_4 + \alg C_j$, $j \ge 0$, form an antichain as well.

Thus, by Lemma~\ref{lem_redpr}, if $(S, \leq)$ contained an infinite antichain, it would contain an infinite antichain of finite projective Heyting algebras, which is impossible by Theorem~\ref{th-prwqo}. Hence, $(S, \leq)$ does not contain infinite antichains and is well founded, which implies that it is well quasi-ordered. This observation completes the proof.
\end{proof}

\begin{theorem}\label{tm-dec}
	There is an algorithm that, given any finite set of equations, decides whether the variety of Heyting algebras defined by these equations is primitive. 
\end{theorem}

Indeed, let $E$ be a finite set of equations. To determine whether the variety~$\VV$ defined by these equations is primitive, we need to check whether all equations from~$E$ hold in some algebra~$\PAlg_i$, $i \in [1,5]$. If none of these algebras satisfy all equations from~$E$, then the variety is primitive; otherwise, it is not primitive.

\section{Conclusions}

The Main Theorem and its corollaries provide a complete characterization of the primitive varieties of Heyting algebras: the variety~$\WW$ is the largest primitive variety; it is locally finite, all its subvarieties are finitely based, and there exists an algorithm that, for any given finite set of equations, decides whether the variety of Heyting algebras defined by these equations is primitive. Much less is known about structurally complete varieties of Heyting algebras.

From the Main Theorem and Corollary~\ref{cor-Med} it follows that $\WW \subseteq \Md$, and thus all primitive varieties are subvarieties of~$\Md$. On the other hand, there are infinitely many subvarieties of~$\Md$ that are structurally complete but not primitive (see~\cite{Skvortsov_Prucnal_1998}), starting with~$\Md$ itself. Moreover, $\Md$ is not locally finite, nor it is finitely based (cf. \cite{Maksimova_Skvortsov_Shehtman}).

The following problems remain open.

\begin{problem}
	Is there a largest structurally complete variety of Heyting algebras?  
\end{problem}

\begin{remark}
	Note, that from \cite[Theorem 1]{Prucnal_Two_1979} it follows that every structurally complete variety is a subvariety of the variety $\KP$ defined by equation $\neg x \to (y \lor z) \approx (\neg x \to 5) \lor (\neg x \to z)$, which is not structurally complete, but contains $\Md$ as a subvariety, but $\KP$ is not structurally complete (cf. \cite{Wojtylak_Problem_2004}).
\end{remark}

\begin{problem}
	Is the class of all structurally complete varieties of Heyting algebras countable?
\end{problem}

\begin{problem}
	Are there structurally complete but non-primitive varieties of Heyting algebras that are not subvarieties of the variety~$\Md$?
\end{problem}

We say that a variety has the \Def{disjunction property} if its free algebra is finitely irreducible. For example, the variety~$\Md$ enjoys the disjunction property, and it was observed in~\cite{Maksimova_Maximal_1986} that $\Md$ is a minimal (with respect to~$\subseteq$) variety with the disjunction property. Hence, except for~$\Md$, all known structurally complete varieties of Heyting algebras lack the disjunction property. 

\begin{problem}
	Is there a structurally complete variety of Heyting algebras that is not a subvariety of~$\Md$?
\end{problem}

\begin{problem}
	Is there a structurally complete variety of Heyting algebras, distinct from~$\mathsf{\Md}$, that enjoys the disjunction property?
\end{problem}
	
	

\end{document}